\documentclass[11pt,letterpaper,reqno]{amsart}

\usepackage[T1]{fontenc}
\usepackage[utf8]{inputenc}
\usepackage{lmodern}
\usepackage{amsmath,amssymb,mathtools,mathrsfs}
\usepackage[protrusion=true,expansion=false]{microtype}
\usepackage[colorlinks=true,linkcolor=blue,citecolor=blue,urlcolor=blue,hypertexnames=false]{hyperref}
\hypersetup{
  pdftitle={Odd Parts of Derivative Period Polynomials: Zero Geometry and a Logarithmic Transition},
  pdfauthor={Seokho Jin},
  pdfsubject={Derivative period polynomials, odd parts, zero geometry, logarithmic transition}
}

\newcommand{\C}{\mathbb C}
\newcommand{\R}{\mathbb R}
\newcommand{\Z}{\mathbb Z}
\newcommand{\T}{\mathbb T}
\newcommand{\SLZ}{\mathrm{SL}_2(\Z)}
\newcommand{\CFI}{Conrey--Farmer--Imamoglu}
\newcommand{\D}{\mathbb D}
\newcommand{\Zset}{\mathcal Z}
\newcommand{\doi}[1]{\href{https://doi.org/#1}{doi:#1}}
\newcommand{\arxiv}[1]{\href{https://arxiv.org/abs/#1}{arXiv:#1}}
\newcommand{\re}{\operatorname{Re}}
\newcommand{\im}{\operatorname{Im}}

\theoremstyle{plain}
\newtheorem{theorem}{Theorem}[section]
\newtheorem{proposition}[theorem]{Proposition}
\newtheorem{lemma}[theorem]{Lemma}
\newtheorem{corollary}[theorem]{Corollary}
\newtheorem{conjecture}[theorem]{Conjecture}
\theoremstyle{remark}
\newtheorem{remark}[theorem]{Remark}
\numberwithin{equation}{section}

\title[Odd derivative period polynomials and transition]
{Odd Parts of Derivative Period Polynomials: Zero Geometry and a Logarithmic Transition}
\author{Seokho Jin}
\address{Department of Mathematics, Chung-Ang University, 84 Heukseok-ro, Dongjak-gu, Seoul 06974, Republic of Korea}
  \email{archimed@cau.ac.kr}
\date{}
\subjclass[2020]{Primary 11F67, 30C15; Secondary 11F11, 11M26, 26C10}
\keywords{derivative period polynomial, odd period polynomial,
self-inversive polynomial, unit-circle zeros, critical $L$-derivatives}

\begin{document}
\raggedbottom

\begin{abstract}
Let $f$ be a normalized level-one Hecke eigenform of even weight $k$, and let
$Q_{f,m}$ be the derivative period polynomial formed from the critical values
of the $m$-th derivative of its completed $L$-function.  We study its odd part
$Q^-_{f,m}(z)=(Q_{f,m}(z)-Q_{f,m}(-z))/2$.

We prove that there is an absolute $K_0$ such that, for every even
$k\ge K_0$, every normalized level-one Hecke eigenform $f$ of weight $k$, and
every integer $m\ge0$, the nonzero zeros of $Q^-_{f,m}$ off the unit circle,
if any, consist of four simple zeros forming a single real reciprocal quartet
$\{\pm b,\pm b^{-1}\}$, where $0<b<1$.  The occurrence and location of this possible
quartet are governed by the critical derivative order
$m_c(k)=(k-1)\log((k-1)/\pi)$.  If $m/m_c(k)\to\theta\in(0,1)$, exactly one
quartet occurs and $b\to(1+\theta)/2$; if $\theta>1$, every nonzero zero is
eventually simple and lies on the unit circle.  At $\theta=1$ the same
real-or-unit-circle containment remains valid, and any quartet that is present
consists of four simple zeros.

For each fixed weight, all nonzero zeros are eventually simple and lie on the
unit circle as $m\to\infty$.  Consequently, the Diamantis--Rolen containment
conjecture holds outside finitely many weight--derivative pairs.  The proof combines an
exact signed-reciprocal completion, uniform split-Mellin saddle estimates yielding
a moving-sine model, and a winding count that transfers disk-zero information
to the unit circle.
\end{abstract}

\maketitle

\section{Introduction}

Period polynomials encode the critical values of modular $L$-functions in
polynomials whose reciprocal symmetry reflects the functional equation.  For
the classical odd period polynomial, every nonzero zero lies on the unit circle
except for one real reciprocal quartet.  Taking an odd part need not preserve
unit-circle zero geometry, so the present odd-part problem is not a formal
consequence of the corresponding full-polynomial result.  For derivative period
polynomials, the derivative order is a genuine geometric parameter: it controls
both the location and the disappearance of the exceptional quartet.  We prove
that, in large weight, the only possible off-circle zeros form a single real
reciprocal quartet and that the transition between this quartet phase and the
unit-circle phase occurs at an explicit critical derivative order of size
$k\log k$.

Let $S_k(\SLZ)$ denote the space of holomorphic cusp forms of weight $k$ for
the full modular group $\SLZ$.  Let $k$ be an even positive integer for which
this space is nonzero, and let
\[
 f(z)=\sum_{n\ge1}a_f(n)e^{2\pi i n z},\qquad a_f(1)=1,
\]
be a normalized Hecke eigenform in this space.  Write
\[
 L_f(s)=\sum_{n\ge1}\frac{a_f(n)}{n^s},\qquad
 \Lambda_f(s)=(2\pi)^{-s}\Gamma(s)L_f(s).
\]
Here and below,
$\Lambda_f^{(m)}(s)=\frac{d^m}{ds^m}\Lambda_f(s)$.  For an integer $m\ge0$,
the derivative period polynomial is
\begin{equation}
 Q_{f,m}(z)=\sum_{j=0}^{k-2}\binom{k-2}{j}i^{1-j}
 \Lambda_f^{(m)}(j+1)z^{k-2-j}.
 \label{eq:Q-definition-intro}
\end{equation}
We study its odd part
\begin{equation}
 Q^-_{f,m}(z)=\frac{Q_{f,m}(z)-Q_{f,m}(-z)}2.
 \label{eq:Qodd-definition-intro}
\end{equation}
The factor $z$ forced by oddness is retained.  When that factor is removed, we
refer explicitly to the \emph{reduced odd part} $Q^-_{f,m}(z)/z$; the two
polynomials have the same nonzero zeros.  Throughout,
\[
 \T=\{z\in\C:|z|=1\},\qquad
 \D=\{z\in\C:|z|<1\},
\]
and, for a nonzero polynomial $P$, $\Zset(P)$ denotes its set of
distinct zeros.  We use $\Zset(P)$ only when $P\not\equiv0$.  A numerical
count of zeros includes multiplicity unless the text says otherwise.
Statements about normalized eigenforms are understood to be vacuous when
$S_k(\SLZ)=\{0\}$.

For $m=0$, the odd period polynomial has a remarkably rigid zero geometry:
it has simple zeros at $0$, $\pm\tfrac12$, and $\pm2$, double zeros at $\pm1$,
and every other zero lies on $\T$ \cite{ConreyFarmerImamoglu}.  Diamantis and
Rolen conjectured that this real-or-unit-circle containment persists for every
derivative order \cite{DiamantisRolen,DiamantisRolenSurvey}.  In the convention
of \eqref{eq:Qodd-definition-intro}, which retains the zero at the origin, their
conjecture takes the following form.

\begin{conjecture}[Diamantis--Rolen odd-part containment]
\label{conj:at-most-four}
For every even weight $k$, every normalized Hecke eigenform $f\in S_k(\SLZ)$,
and every integer $m\ge0$, the polynomial $Q^-_{f,m}$ is nonzero, and there is
a real number $b=b(f,m,k)>0$ such that
\begin{equation}
 \Zset(Q^-_{f,m})
 \subseteq \T\cup\{0,\pm b,\pm b^{-1}\}.
 \label{eq:DR-odd-containment}
\end{equation}
The exceptional real reciprocal orbit may be absent.
\end{conjecture}

\subsection*{Main results}
We first describe the large-weight zero geometry uniformly in the derivative
order and the eigenform.

\begin{samepage}
\begin{theorem}[Odd-part containment in large weight]
\label{thm:intro-main}
There exists an absolute constant $K_0$ such that, for every even weight
$k\ge K_0$, every integer $m\ge0$, and every normalized Hecke eigenform
$f\in S_k(\SLZ)$, the polynomial $Q^-_{f,m}$ is nonzero, and its nonzero zeros
off $\T$, if any, consist of four simple zeros forming a single real reciprocal
quartet
\[
 \{\pm b_{f,m,k},\ \pm b_{f,m,k}^{-1}\},
 \qquad 0<b_{f,m,k}<1.
\]
\end{theorem}
\end{samepage}

Theorem~\ref{thm:intro-main} is uniform in $m$: the derivative order may grow
arbitrarily with the weight, including through the transition at which the
quartet collides with the unit circle.
The theorem does not assert the presence of the quartet in that collision
regime, but whenever the quartet occurs, all four of its zeros are simple.

The possible quartet is governed more precisely by the following explicit
critical derivative order:
\begin{equation}
 m_c(k):=(k-1)\log\frac{k-1}{\pi}\asymp k\log k.
 \label{eq:intro-critical-scale}
\end{equation}

\begin{theorem}[Logarithmic transition in the derivative order]
\label{thm:intro-logarithmic-transition}
Let $k\to\infty$ through even weights, let $m=m(k)\in\Z_{\ge0}$, and suppose
that
\[
 \frac{m(k)}{m_c(k)}\longrightarrow\theta
\]
for some $0<\theta<\infty$.  The following conclusions hold uniformly over
normalized level-one Hecke eigenforms.

If $0<\theta<1$, then, for all sufficiently large $k$, the polynomial
$Q^-_{f,m}$ has a simple zero at $0$, and its nonzero off-circle zeros form exactly one
real reciprocal quartet whose four zeros are simple.  All remaining nonzero
zeros are simple and lie on $\T$.  If
$b_{f,m,k}\in(0,1)$ denotes the positive member of the quartet inside $\D$,
then
\[
 b_{f,m,k}\longrightarrow\frac{1+\theta}{2},\qquad
 b_{f,m,k}^{-1}\longrightarrow\frac{2}{1+\theta}.
\]

If $\theta>1$, then, for all sufficiently large $k$, the origin is simple and
every nonzero zero of $Q^-_{f,m}$ is simple and lies on $\T$.

If $\theta=1$, then $Q^-_{f,m}$ is nonzero and the containment
\eqref{eq:DR-odd-containment} holds for all sufficiently large $k$; if the
exceptional quartet is present, its four zeros are simple.  This critical
statement does not decide whether the quartet is present inside the collision
window.
\end{theorem}

Thus, on the macroscopic scale, the quartet is a pre-critical phenomenon.  As
$\theta$ increases through $(0,1)$, its positive zero inside $\D$ moves from
the limiting position $1/2$ toward $1$; once $\theta>1$, no off-circle zero
remains.  The critical ratio $\theta=1$ is the point at which the moving zeros of
the sine model reach the unit circle.

\subsection*{Refinements of the transition}
The ratio limit does not describe the width of the collision region.  For sequences
$m(k)\ge1$, Corollary~\ref{cor:logarithmic-window} shows that, on each one-sided
regime with $|m(k)-m_c(k)|/\log k\to\infty$, the sign of $m(k)-m_c(k)$
determines the phase; the case $m=0$ is covered by the classical theorem above.
More precisely, if
\[
 \frac{m_c(k)-m(k)}{\log k}\longrightarrow+\infty,
\]
then, for all sufficiently large $k$, exactly one real reciprocal quartet is
present off $\T$, and its four zeros are simple; if
\[
 \frac{m(k)-m_c(k)}{\log k}\longrightarrow+\infty,
\]
then every nonzero zero is simple and lies on $\T$.  In the near-critical
pre-critical subrange $m_c(k)-m(k)=o(k\log k)$, the inner positive zero satisfies
\begin{equation}
 1-b_{f,m,k}
 \sim
 \frac{m_c(k)-m(k)}
 {(k-1)\bigl(1+2\log((k-1)/\pi)\bigr)}.
 \label{eq:intro-precritical-location}
\end{equation}
The internal scaling law for $m-m_c(k)=O(\log k)$ remains open; throughout that
window Theorem~\ref{thm:intro-main} still gives real-or-unit-circle
containment, and all four zeros of any exceptional quartet are simple.

At the opposite end of the pre-critical phase, where $m$ is fixed or at most
logarithmic, the quartet begins near the classical points
$\{\pm\tfrac12,\pm2\}$.  Theorem~\ref{thm:alpha-one-endpoint} gives the
following first-order refinement.  Let $\psi=\Gamma'/\Gamma$ and define the
endpoint displacement
\begin{equation}
 \mathfrak d_{k,m}:=\frac{m\,\psi'(k-1)}{\psi(k-1)-\log(2\pi)}.
 \label{eq:intro-upper-displacement}
\end{equation}
For each fixed $B>0$, uniformly over integers $1\le m\le B\log k$ and
normalized level-one Hecke eigenforms $f$,
\begin{equation}
 b_{f,m,k}=\frac12+\frac12\mathfrak d_{k,m}
            +o_B(\mathfrak d_{k,m}),\qquad
 b_{f,m,k}^{-1}=2-2\mathfrak d_{k,m}
            +o_B(\mathfrak d_{k,m}).
 \label{eq:intro-upper-root}
\end{equation}
Both $o_B$-terms are uniform in $m$ and $f$ in this range.
Since $\mathfrak d_{k,m}\asymp_B m/(k\log k)$, this gives the corresponding
first-order formula for each fixed $m\ge1$.  For $m=0$, the corresponding
classical off-circle zeros are exactly $\pm\tfrac12$ and $\pm2$ in the
normalization used here; the unit-circle zeros at $\pm1$ are double.

\subsection*{Fixed weight and finite exceptions}
A different asymptotic regime is needed to control the remaining bounded
weights.  For fixed $k$, all nonzero zeros eventually lie on $\T$ as
$m\to\infty$.

\begin{theorem}[Fixed weight and large derivative order]
\label{thm:intro-fixed-weight-large-derivative}
Fix an even weight $k$ and a normalized Hecke eigenform $f\in S_k(\SLZ)$.  Then
$\Lambda_f^{(m)}(2)\ne0$ for all sufficiently large $m$, and
\begin{equation}
 \frac{Q^-_{f,m}(z)}{(k-2)\Lambda_f^{(m)}(2)z}
 =z^{k-4}+(-1)^m+o_{\mathrm{coeff}}(1)
 \qquad(m\to\infty).
 \label{eq:intro-fixed-weight-root-unity-limit}
\end{equation}
Here $o_{\mathrm{coeff}}(1)$ means coefficientwise convergence to zero in this
fixed degree.  Together with the real coefficients and signed reciprocity of
the normalized reduced polynomial, this implies that, for all sufficiently
large $m$, the origin is a simple zero of $Q^-_{f,m}$ and every nonzero zero is
simple and lies on $\T$.
Moreover, for each fixed even weight $k$ there is an integer $m_0(k)$ for which
this conclusion holds for every normalized Hecke eigenform in $S_k(\SLZ)$ and
every $m\ge m_0(k)$.
\end{theorem}

Combining the large-weight and fixed-weight theorems gives the following
finite-exception form of Conjecture~\ref{conj:at-most-four}.

\begin{corollary}[Exclusion of unbounded exceptional families]
\label{cor:finite-pair-exception}
There is a finite set
\[
 \mathcal P\subset 2\Z_{>0}\times\Z_{\ge0}
\]
such that, for every even weight $k$, every normalized Hecke eigenform
$f\in S_k(\SLZ)$, and every integer $m\ge0$, the polynomial $Q^-_{f,m}$ is
nonzero and the containment \eqref{eq:DR-odd-containment} holds whenever
$(k,m)\notin\mathcal P$.  Consequently, only finitely many triples $(k,f,m)$
can fail this conclusion.
\end{corollary}

Thus Corollary~\ref{cor:finite-pair-exception} proves the Diamantis--Rolen
odd-part containment conjecture outside a finite set of weight--derivative
pairs.  In particular, no counterexample can occur along an unbounded sequence
of weight--derivative pairs.  The corollary does not assert that the finite
remaining region is empty.

\subsection*{Outline of the proof}
The proof proceeds in four steps.

\medskip
\noindent\emph{Algebraic reduction (Section~\ref{sec:normalization-completion}).}
After the leading right-edge coefficient $c_0$ is shown to be nonzero uniformly
in the large-weight range, the functional equation expresses $Q^-_{f,m}$ as a
nonzero real multiple of the exact odd signed-reciprocal completion
\begin{equation}
 P_{f,m,k}(z):=p_{f,m,k}(z)+(-1)^m z^{k-2}p_{f,m,k}(1/z),
 \label{eq:intro-completion}
\end{equation}
where $p_{f,m,k}$ is a real odd polynomial formed from the right-hand half of
the critical-derivative array.  The completion has real coefficients, is odd,
and satisfies the signed reciprocity relation
\[
 P_{f,m,k}(z)=(-1)^m z^{k-2}P_{f,m,k}(1/z).
\]
Thus its set of nonzero zeros is invariant under conjugation, sign change, and
inversion.
A real nonzero off-circle zero generates a reciprocal quartet.  At the critical
collision, the later localization near $\pm1$ rules out all orbit collapses, so
a nonreal off-circle zero would generate eight distinct zeros.
Signed reciprocity alone does not force zeros onto the unit circle.  The two
winding lemmas in Section~\ref{sec:normalization-completion} supply the missing
quantitative information: they turn a count of the zeros of $p_{f,m,k}$ in
$\D$ into a lower bound for the number of zeros of its completion on $\T$.
Thus the zero problem for $Q^-_{f,m}$ is reduced to a disk-zero count for the
one-sided polynomial.

\medskip
\noindent\emph{Uniform right-edge asymptotics
(Section~\ref{sec:right-edge-estimates}).}
A split Mellin formula reduces the right-edge coefficients to the model moments
\[
 \mathcal J_m(S)=\int_0^\infty e^{Sx-2\pi e^x}x^m\,dx.
\]
At $S=k-1$, let $x_0=x_0(k,m)>0$ denote the unique saddle, so that
\[
 2\pi e^{x_0}=k-1+\frac{m}{x_0},
\]
and define
\begin{equation}
 \alpha_{k,m}:=\frac{k-1}{2\pi e^{x_0}}
 =\frac{k-1}{k-1+m/x_0}\in(0,1].
 \label{eq:intro-alpha}
\end{equation}
Under the real interpolation in $m$, the parameter $\alpha_{k,m}$ decreases
continuously from $1$ at $m=0$ toward $0$ and crosses $1/2$ exactly when
$m=m_c(k)$.  The local saddle analysis controls the first $O(\log k)$
right-edge coefficients.  Since the degree grows with $k$, the global
factorial-tail bound is needed to pass from these local ratios to a uniform
approximation of the whole edge polynomial.  Together, the estimates proved
in Section~\ref{sec:right-edge-estimates} show
that, on each fixed compact set, $p_{f,m,k}$ is approximated uniformly over
$m\ge0$ and normalized eigenforms $f$ by the moving-sine model
\[
 z\longmapsto \sin(2\pi\alpha_{k,m}z).
\]
The value $\alpha=1/2$ is the collision point of the model: for
$0<\alpha<1/2$ it has only the origin in $\D$, whereas for
$1/2<\alpha<1$ it has two additional disk zeros; at $\alpha=1/2$ those zeros
reach $\pm1$.  The endpoints $\alpha_{k,m}\to0$ and
$\alpha_{k,m}\to1$ require separate refinements.

\medskip
\noindent\emph{Zero geometry and the three exceptional regimes
(Sections~\ref{sec:moving-phases}--\ref{sec:upper-endpoint}).}
The boundary size of the model is
\begin{equation}
 \min_{|z|=1}|\sin(2\pi\alpha z)|
 \asymp
 \min\{\alpha,\ |2\alpha-1|,\ 1-\alpha\},
 \qquad 0<\alpha<1,
 \label{eq:intro-sine-boundary-margin}
\end{equation}
as proved in Lemma~\ref{lem:sine-boundary-margin}.  Proposition~\ref{prop:transition-moving-edge-quantitative}
gives an absolute approximation error of order $k^{-1}$ on fixed compact
sets.  Hence the ordinary Rouch\'e argument works whenever
\begin{equation}
 k\min\{\alpha_{k,m},\ |2\alpha_{k,m}-1|,\ 1-\alpha_{k,m}\}
 \longrightarrow\infty.
 \label{eq:intro-zeroth-order-separated-range}
\end{equation}
The separated phase and the collision at $\alpha_{k,m}=1/2$ are treated in
Section~\ref{sec:moving-phases}.  At the collision, annular exclusion localizes
every nonzero off-circle zero near $\pm1$.  A nonreal zero there would generate eight
distinct zeros under conjugation, sign change, and inversion, whereas the
boundary-sensitive winding count leaves room for at most four.

The two endpoint sections address the other terms in the minimum in
\eqref{eq:intro-sine-boundary-margin}.  When $\alpha_{k,m}\to0$, the sine model
itself collapses in amplitude, and Section~\ref{sec:lower-endpoint} replaces the
absolute approximation by a relative estimate for the full edge array.  When
$\alpha_{k,m}\to1$, the displacement of the model zeros is of the same order
as, or smaller than, the zeroth-order error.  Section~\ref{sec:upper-endpoint}
therefore uses the first Gamma-factor correction; the uniform Bell-polynomial
estimates needed for that correction are proved in
Appendix~\ref{sec:bell-appendix}.

\medskip
\noindent\emph{The derivative-order scale and the fixed-weight limit
(Sections~\ref{sec:derivative-axis-consequences}
and~\ref{sec:fixed-weight-large-derivative}).}
Section~\ref{sec:derivative-axis-consequences} translates the saddle parameter
$\alpha_{k,m}$ into the critical order $m_c(k)$ and proves the transition theorem
and the uniform large-weight theorem.  Section~\ref{sec:fixed-weight-large-derivative}
uses a different mechanism: for fixed $k$, the split-Mellin moment at $k-2$
dominates those at smaller parameters, so the normalized reduced odd part in
\eqref{eq:intro-fixed-weight-root-unity-limit} tends coefficientwise to
$z^{k-4}+(-1)^m$.  This gives the fixed-weight theorem and, together with
the large-weight result, Corollary~\ref{cor:finite-pair-exception}.  The scope of
effectivity is recorded in Appendix~\ref{sec:effectivity}.

\subsection*{Related work}
The classical theory of period polynomials goes back to Manin,
Kohnen--Zagier, and Zagier \cite{Manin,KohnenZagier,ZagierPeriods}.  The
unit-circle problem for ordinary period polynomials was developed in
\cite{ConreyFarmerImamoglu,ElGuindyRaji,JinMaOnoSoundararajan}, and higher-level
odd and even parts, zero counts, and interlacing have subsequently been studied
in \cite{ChoiOdd,ChoiEven,Oh,KoMackenzieRossXue,KoMackenzieXue}.  For
derivative period polynomials, Diamantis and Rolen developed the cohomological
framework, formulated the full and parity conjectures, and proved an
Eisenstein-series case \cite{DiamantisRolen,DiamantisRolenSurvey}; see also
Im--Kim for a first-derivative result in higher level \cite{ImKim}.  Related
reciprocal zero phenomena occur for Ramanujan-type polynomials and Hilbert
modular forms \cite{MurtySmythWang,LalinSmyth,BabeiRolenWagner}.  The present
paper treats the zero geometry of the odd derivative family uniformly in the
joint weight--derivative aspect and identifies the transition law governing the
possible real quartet.

A separate preprint \cite{JinCircularSampling} treats the full derivative
period polynomial by circular sampling and multiplier theory.  That result
does not imply the present one, since taking odd parts need not preserve
unit-circle zeros: for example, the odd part of $(1+z)^6$ has nonzero zeros of
moduli $3^{-1/2}$ and $3^{1/2}$.  No result from that preprint is used here.

\subsection*{Scope}
The moving-edge estimate and the final large-weight theorem are uniform in all
integers $m\ge0$ and normalized level-one Hecke eigenforms; the separated,
critical, and endpoint estimates are uniform in their stated ranges.  The level-one
hypothesis enters
through the functional equation, the Fourier-coefficient bounds, and the
conductor-one Mellin and Dirichlet tails.  The completion and winding arguments
are more flexible, but no higher-level theorem is claimed in this paper.

Three questions remain beyond the present argument.  First, the proof confines
every possible exception to a finite low-weight, bounded-derivative region but does
not certify that the region is empty.  Second, it locates the transition to an
$O(\log k)$ window without deriving a complete scaled collision law inside that
window.  Third, the constants are not optimized to give practical numerical
values of $K_0$ or the fixed-weight thresholds $m_0(k)$.

We begin in Section~\ref{sec:normalization-completion} with an exact right-edge
completion and a winding principle that together reduce the zero problem to the
disk zeros of a one-sided polynomial.

\section{Right-edge normalization and signed-reciprocal completion}
\label{sec:normalization-completion}

This section isolates the algebraic reduction used throughout the zero
analysis.  Normalizing the right-edge coefficients turns the functional
equation into a signed reflection relation; rotation and odd extraction then
produce an exact odd signed-reciprocal completion.  Two winding lemmas convert
disk-zero counts for the edge polynomial into unit-circle counts for the
completion.  The final subsection identifies the case $m=0$ with the classical
normalization.

\subsection{Right-edge data}

Let $k$ be even and let $m\ge0$ be an integer.  If $S_k(\SLZ)=\{0\}$, the
statements involving an eigenform of weight $k$ are vacuous.  Otherwise fix a
normalized level-one Hecke eigenform $f$ of weight $k$, written
\[
 f(z)=\sum_{n\ge1}a_f(n)e^{2\pi i n z},\qquad a_f(1)=1.
\]
Throughout, put
\begin{align*}
 d&=k-2=2M, & s_0&=k-1,\\
 G(s)&=(2\pi)^{-s}\Gamma(s), &
 \mathcal D_{f,m}(s)&=\frac{\Lambda_f^{(m)}(s)}{G(s)}.
\end{align*}
The level-one functional equation gives
\begin{equation}
 \Lambda_f^{(m)}(s)=(-1)^{k/2+m}\Lambda_f^{(m)}(k-s).
 \label{eq:level-one-derivative-FE}
\end{equation}

With $r=0$ corresponding to the rightmost critical point $s=s_0=k-1$, define
coefficients
\begin{equation}
 c_r=\mathcal D_{f,m}(s_0-r)\frac{(2\pi)^r}{r!},
 \qquad 0\le r\le 2M,
 \label{eq:cr-definition}
\end{equation}
and set
\begin{equation}
 \mathcal R_{f,m,k}(w)=\sum_{r=0}^{2M}c_rw^r.
 \label{eq:R-mk}
\end{equation}
The normalization is chosen so that the functional equation becomes the
signed reflection relation \eqref{eq:c-reciprocity} below.  Reindexing
\eqref{eq:Q-definition-intro} gives
\begin{equation}
 Q_{f,m}(z)=i^{1-2M}(2M)!(2\pi)^{-2M-1}\mathcal R_{f,m,k}(iz).
 \label{eq:Q-R-rotation}
\end{equation}
Moreover \eqref{eq:level-one-derivative-FE} implies
\begin{equation}
 c_{2M-r}=\eta_{k,m}c_r,
 \qquad \eta_{k,m}=(-1)^{k/2+m}.
 \label{eq:c-reciprocity}
\end{equation}
Put
\begin{equation}
 q_{f,m,k}(w)=\sum_{r=0}^{M-1}c_rw^r+\frac12c_Mw^M.
 \label{eq:q-edge}
\end{equation}
Then
\begin{equation}
 \mathcal R_{f,m,k}(w)=q_{f,m,k}(w)+\eta_{k,m}w^{2M}q_{f,m,k}(1/w).
 \label{eq:R-edge-decomposition}
\end{equation}
At the central index this uses $c_M=\eta_{k,m}c_M$ from
\eqref{eq:c-reciprocity}; in particular, $c_M=0$ when $\eta_{k,m}=-1$.
The factor $1/2$ in \eqref{eq:q-edge} avoids double counting the central
coefficient in \eqref{eq:R-edge-decomposition}; the rotation $w=iz$ removes the
phases in $Q_{f,m}$, and odd extraction isolates the relevant parity.  Whenever
$c_0\ne0$, define the odd edge polynomial
\begin{equation}
 p_{f,m,k}(z)=\frac{q_{f,m,k}(iz)-q_{f,m,k}(-iz)}{2ic_0}.
 \label{eq:odd-edge-p}
\end{equation}
This normalization is conditional on $c_0\ne0$, established uniformly in all
later large-weight applications by
Proposition~\ref{prop:transition-near-edge-ratio}.  Define the normalized
signed-reciprocal completion, with completion parameter $2M$, by
\begin{equation}
 P_{f,m,k}(z):=p_{f,m,k}(z)+(-1)^m z^{2M}p_{f,m,k}(1/z).
 \label{eq:normalized-odd-completion}
\end{equation}
Since $p_{f,m,k}$ is odd of degree at most $M$, the reflected term
$z^{2M}p_{f,m,k}(1/z)$ is a polynomial of degree at most $2M-1$; hence $2M$ is
a completion parameter, not necessarily the actual degree.  The definition
gives
\begin{equation}
 P_{f,m,k}(z)=(-1)^m z^{2M}P_{f,m,k}(1/z).
 \label{eq:P-reciprocity}
\end{equation}
Thus $q_{f,m,k}$ is the one-sided right-edge polynomial, $p_{f,m,k}$ is its
rotated odd part, and $P_{f,m,k}$ is the corresponding signed-reciprocal
completion.

We use the standard reality of the Hecke eigenvalues.  Self-adjointness gives
$a_f(n)\in\R$, so $L_f(\bar s)=\overline{L_f(s)}$.  By analytic continuation,
$\Lambda_f^{(m)}(s)$ and $\mathcal D_{f,m}(s)$ are real for $s>0$.  Thus $c_r$,
$\mathcal R_{f,m,k}$, $q_{f,m,k}$, and $p_{f,m,k}$ have real coefficients
whenever they are defined.  The same is true of $P_{f,m,k}$, and $p_{f,m,k}$ is
odd by construction.  Independently, only odd indices $j$ occur in $Q^-_{f,m}$,
with $i^{1-j}\in\{\pm1\}$, so $Q^-_{f,m}\in\R[z]$.

\begin{lemma}[Odd-part completion]
\label{lem:odd-part-completion}
Assume $c_0\ne0$.  There is a nonzero real constant $C_{f,m,k}$ such that
\begin{equation}
 Q_{f,m}^{-}(z)=C_{f,m,k}P_{f,m,k}(z).
 \label{eq:odd-completion}
\end{equation}
In particular, $Q^-_{f,m}\not\equiv0$ if and only if
$P_{f,m,k}\not\equiv0$, and in that case they have the same zeros, with the
same multiplicities.  The completion $P_{f,m,k}$ is odd.  If $M\ge2$ and
$p_{f,m,k}'(0)\ne0$, then $P_{f,m,k}$ has degree $2M-1$ and a simple zero at
the origin.
\end{lemma}

\begin{proof}
Let $q^-(w)=(q_{f,m,k}(w)-q_{f,m,k}(-w))/2$.  Since $2M$ is even,
\eqref{eq:R-edge-decomposition} gives
\[
 \frac{\mathcal R_{f,m,k}(w)-\mathcal R_{f,m,k}(-w)}2
 =q^-(w)+\eta_{k,m}w^{2M}q^-(1/w).
\]
By definition,
\[
 q^-(iz)=ic_0p_{f,m,k}(z),
 \qquad q^-(1/(iz))=-ic_0p_{f,m,k}(1/z),
\]
where the second identity uses oddness.  Finally,
$-\eta_{k,m}(-1)^M=(-1)^m$, because $k/2=M+1$.  Combining these identities with
\eqref{eq:Q-R-rotation} proves \eqref{eq:odd-completion}.

Since $p_{f,m,k}$ is odd and has degree at most $M<2M$, the polynomial
$P_{f,m,k}$ contains only nonnegative odd powers of $z$.  If $M\ge2$ and
$p_{f,m,k}'(0)=a_1\ne0$, then the reflected term contributes the nonzero
coefficient $(-1)^m a_1$ of $z^{2M-1}$, so the degree is $2M-1$.  The
reflected term contributes no $z$-term, because all its exponents are at least
$M\ge2$; hence the coefficient of $z$ is $a_1$, and the zero at the origin is
simple.
\end{proof}

Although the normalized completion requires $c_0\ne0$, the following reality
and signed reciprocity of $Q^-_{f,m}$ are intrinsic.

\begin{lemma}[Reality and signed reciprocity of the odd derivative period polynomial]
\label{lem:Qodd-real-reciprocal}
For every normalized level-one Hecke eigenform $f$ and every integer $m\ge0$, the
polynomial $Q_{f,m}^{-}$ has real coefficients and satisfies
\begin{equation}
 Q_{f,m}^{-}(z)=(-1)^m z^{k-2}Q_{f,m}^{-}(1/z).
 \label{eq:Qodd-reciprocity-direct}
\end{equation}
\end{lemma}

\begin{proof}
The coefficients are real by the preceding observation.  Let $d=k-2$.  Pair
the odd index $j$ with $d-j$, which is again odd.  The
functional equation gives
\[
 \Lambda_f^{(m)}(d-j+1)=(-1)^{k/2+m}\Lambda_f^{(m)}(j+1).
\]
A direct comparison of the two paired coefficients yields
\[
 \binom dj i^{1-d+j}\Lambda_f^{(m)}(d-j+1)
 =(-1)^m\binom dj i^{1-j}\Lambda_f^{(m)}(j+1),
\]
because $k/2=M+1$ and $j$ is odd.  This coefficient identity is exactly
\eqref{eq:Qodd-reciprocity-direct}.
\end{proof}

\subsection{Winding counts for signed-reciprocal completions}

To distinguish angular variables from the macroscopic ratio parameter $\theta$,
we write angles as $\vartheta$ throughout the paper.  Here $N$ denotes the
completion parameter, the exponent of the reflected term,
and need not equal the actual degree of $P_{\sigma,N}$.  The following
polynomial lemmas convert disk-zero counts for $p$ into unit-circle lower bounds
for its signed-reciprocal completion.  The first assumes a zero-free boundary;
the second retains boundary zeros with their multiplicities.

\begin{lemma}[Unit-circle count from disk zeros]
\label{lem:disk-zero-circle-count}
Let $N\ge1$ be an integer, let $p\in\R[z]$ have degree less than $N$, and suppose that
$p$ has no zero on $\T$.  Let $r$ be the number of zeros of $p$ in $\D$, counted
with multiplicity.  For $\sigma\in\{\pm1\}$, assume that
\[
 P_{\sigma,N}(z)=p(z)+\sigma z^Np(1/z)
\]
is not identically zero.  If $r\le N/2$, then $P_{\sigma,N}$ has at least
$N-2r$ distinct zeros on $\T$.
\end{lemma}

\begin{proof}
On $z=e^{i\vartheta}$ put $H(\vartheta)=e^{-iN\vartheta/2}p(e^{i\vartheta})$.  The function
$H$ never vanishes.  By the argument principle, a continuous lift
$\phi(\vartheta)=\arg H(\vartheta)$ satisfies
\[
 \phi(\vartheta+2\pi)-\phi(\vartheta)=2\pi r-N\pi=-\pi(N-2r).
\]
Moreover,
\[
 e^{-iN\vartheta/2}P_{1,N}(e^{i\vartheta})=2\re H(\vartheta),
 \qquad
 e^{-iN\vartheta/2}P_{-1,N}(e^{i\vartheta})=2i\im H(\vartheta).
\]
Let $L=N-2r$.  If $L=0$ there is nothing to prove.  For $\sigma=1$ the zeros
on $\T$ are the points at which $\phi(\vartheta)$ meets the lattice
$\pi/2+\pi\Z$, while for $\sigma=-1$ they are the points at which
$\phi(\vartheta)$ meets $\pi\Z$.  Let $\mathcal L_\sigma$ denote the relevant
lattice.  Put $a=\phi(0)$ and $b=\phi(2\pi)=a-L\pi$.  The half-open interval
$I_\phi=(b,a]$ has length $L\pi$ and therefore contains exactly $L$ points of
$\mathcal L_\sigma$.  Its endpoint convention includes the initial value
$\phi(0)$ and excludes the terminal value $\phi(2\pi)$, so a zero represented
by both $\vartheta=0$ and $\vartheta=2\pi$ is not counted twice.  For every
$\ell\in I_\phi\cap\mathcal L_\sigma$, the intermediate value theorem gives a
point $\vartheta_\ell\in[0,2\pi)$ with $\phi(\vartheta_\ell)=\ell$.  Distinct
levels give distinct points $e^{i\vartheta_\ell}$.  Hence $P_{\sigma,N}$ has at
least $L=N-2r$ distinct zeros on $\T$.
\end{proof}

\begin{lemma}[Unit-circle count with boundary zeros]
\label{lem:disk-zero-circle-count-boundary}
Let $N\ge1$ be an integer, let $p\in\R[z]$ have degree less than $N$, and put
\[
 P_{\sigma,N}(z)=p(z)+\sigma z^Np(1/z),\qquad \sigma\in\{\pm1\}.
\]
Assume that $P_{\sigma,N}$ is not identically zero.  Let $r$ be the number of
zeros of $p$ in the open unit disk and let $t$ be the number of zeros of $p$ on
$\T$, both counted with multiplicity.  If $r\le (N-t)/2$, then
$P_{\sigma,N}$ has at least $N-2r$ zeros on $\T$, counted with multiplicity.  In
particular, if $p$ has at most $R$ zeros in $\overline\D$ and $R\le N/2$, then
$P_{\sigma,N}$ has at least $N-2R$ zeros on $\T$.
\end{lemma}

\begin{proof}
Factor $p=p_Tp_0$, where $p_T$ is the monic factor whose zeros are precisely
the zeros of $p$ on $\T$, with multiplicity.  Since $p$ has real coefficients,
$z^t p_T(1/z)=\chi_T p_T(z)$ for some real sign
$\chi_T\in\{\pm1\}$.  Hence
\[
 P_{\sigma,N}(z)=p_T(z)\left(p_0(z)+\sigma\chi_T z^{N-t}p_0(1/z)\right).
\]
The polynomial $p_0$ has no zeros on $\T$ and has exactly $r$ zeros in $\D$.
Applying Lemma~\ref{lem:disk-zero-circle-count} to $p_0$ with completion
parameter $N-t$ gives at least $N-t-2r$ zeros on $\T$.  The factor $p_T$
contributes the remaining $t$ unit-circle zeros, counted with multiplicity.
For the final assertion, $r+t\le R\le N/2$ gives $r\le(N-t)/2$, and the
preceding lower bound $N-2r$ is at least $N-2R$.
\end{proof}

\begin{remark}[Effectivity]
The winding lemmas are exact; quantitative loss enters only through the later
zero estimates for the edge polynomial $p$.
\end{remark}

\subsection{Compatibility with the classical normalization}

\begin{lemma}[The case $m=0$ and the classical odd period polynomial]
\label{lem:mzero-classical-period}
Let
\[
 r_f(z):=\int_0^{i\infty} f(\tau)(\tau-z)^{k-2}\,d\tau
\]
be the classical period polynomial in the integral normalization used here, and
write $r_f^{-}(z)=(r_f(z)-r_f(-z))/2$.  Then
\begin{equation}
 Q_{f,0}(z)=r_f(z),
 \qquad Q_{f,0}^{-}(z)=r_f^{-}(z).
 \label{eq:mzero-equals-classical-period}
\end{equation}
With the coefficient convention displayed in \CFI~\cite[formula (1.1)]{ConreyFarmerImamoglu},
their odd period polynomial is precisely the polynomial $r_f^{-}$ above.  No
rotation or change of variable is involved in this comparison.  Applying
\cite[Theorem 1.1]{ConreyFarmerImamoglu} in this normalization gives simple
zeros at $0$, $\pm2$, and $\pm2^{-1}$, double zeros at $\pm1$, and places every
other zero on $\T$.  In particular, the containment needed here is
\[
 \Zset(Q_{f,0}^{-})=\Zset(r_f^{-})\subseteq \T\cup\{0,\pm 2,\pm 2^{-1}\}.
\]
\end{lemma}

\begin{proof}
Expanding the integral and putting $\tau=iy$ gives
\begin{align*}
 r_f(z)
 &=\sum_{j=0}^{k-2}\binom{k-2}{j}(-z)^{k-2-j}
   \int_0^{i\infty} f(\tau)\tau^j\,d\tau  \\
 &=\sum_{j=0}^{k-2}\binom{k-2}{j}(-1)^j i^{j+1}
   \Lambda_f(j+1)z^{k-2-j}.
\end{align*}
Since $k-2$ is even, $(-1)^ji^{j+1}=i^{1-j}$ for every $j$.  This is precisely
\eqref{eq:Q-definition-intro} with $m=0$, proving both identities in
\eqref{eq:mzero-equals-classical-period}.  For odd $j$, the coefficient is
\[
\begin{aligned}
 &(-1)^{(j-1)/2}\binom{k-2}{j}\Lambda_f(j+1)\\
 &\qquad={(-1)^{(j-1)/2}\binom{k-2}{j}}
   j!(2\pi)^{-j-1}L_f(j+1).
\end{aligned}
\]
This is the coefficient convention in
\cite[formula (1.1)]{ConreyFarmerImamoglu}.  The associated monomial is the same
$z^{k-2-j}$, so no substitution such as $z\mapsto iz$ is involved.
Multiplication of the entire period polynomial by a nonzero scalar does not
change its zero set.
\end{proof}

\section{Uniform right-edge estimates}
\label{sec:right-edge-estimates}

This section proves the analytic approximation used in every later large-weight
zero argument.  The saddle analysis controls the first $O(\log k)$ right-edge
coefficients; because the degree grows with $k$, a separate factorial majorant
is required for the remaining coefficients.  Together these estimates yield an
exponential model for $q_{f,m,k}$, and hence a sine model for $p_{f,m,k}$,
uniformly in the integer derivative order $m\ge0$ and in the normalized
eigenform $f$.

We use the following split-Mellin notation throughout the paper: for
$S\in\R$, put
\begin{equation}
 \mathcal I_{f,m}(S):=\int_1^\infty f(iy)y^S(\log y)^m\,\frac{dy}{y}.
 \label{eq:log-moment}
\end{equation}
We also use the corresponding model moments
\begin{equation}
 \mathcal J_{n,m}(S):=\int_1^\infty e^{-2\pi n y}y^S(\log y)^m\,\frac{dy}{y},
 \qquad \mathcal J_m(S):=\mathcal J_{1,m}(S).
 \label{eq:model-moments}
\end{equation}
These definitions will also be used later in the fixed-weight argument.

The proof first establishes local model ratios and the weighted tail needed at
$\alpha\to0$, then transfers the ratios to modular $L$-derivatives through the
split-Mellin formula, and finally controls the remaining coefficients by a
global factorial envelope.  Unless a statement says otherwise, every implied
constant in this section is independent of $m$ and $f$; dependence on fixed
auxiliary parameters such as $C$, $R$, and $A$ is displayed.  The first-order
Gamma expansion needed at the upper endpoint is isolated in
Appendix~\ref{sec:bell-appendix}.

\subsection{The saddle parameter and the local Mellin model}

This subsection isolates the one-dimensional saddle which controls the right
edge before the cusp-form Fourier coefficients are reintroduced.

Put $s_0=k-1$.  The function
$x\mapsto 2\pi e^x-m/x$ is strictly increasing on $x>0$, so in the present
range $s_0>2\pi$ there is a unique solution $x_0=x_0(k,m)>0$ of
\begin{equation}
 2\pi e^{x_0}=s_0+\frac{m}{x_0}.
 \label{eq:transition-saddle}
\end{equation}
Define
\begin{equation}
 \alpha_{k,m}:=\frac{s_0}{2\pi e^{x_0}}
 =\left(1+\frac{m}{s_0x_0}\right)^{-1}\in(0,1].
 \label{eq:transition-alpha}
\end{equation}
The change of variables $y=e^x$ shows that the saddle is governed by the Mellin
parameter $S$, rather than by the exponent $S-1$ in $y^{S-1}\,dy$.

\begin{proposition}[Uniform split-Mellin near-edge ratio]
\label{prop:transition-near-edge-ratio}
Fix $C>0$.  For all sufficiently large even weights $k$, uniformly for all
integers $m\ge0$, all normalized Hecke eigenforms $f\in S_k(\SLZ)$, and all
integers $0\le r\le C\log k$, the denominator $\mathcal D_{f,m}(s_0)$ is nonzero and
\begin{equation}
 \frac{\mathcal D_{f,m}(s_0-r)}{\mathcal D_{f,m}(s_0)}
 =\left(\alpha_{k,m}\right)^r
 \left(1+O_C\!\left(\frac{(1+r)^2}{k}\right)\right).
 \label{eq:transition-ratio-proved}
\end{equation}
\end{proposition}

We first establish two model estimates.  In this auxiliary analysis, the
integral formula in \eqref{eq:model-moments} is extended to real $m\ge0$.  For
$S>2\pi$, the change of variables $y=e^x$ gives
\begin{equation}
 \mathcal J_m(S)=\int_0^\infty e^{Sx-2\pi e^x}x^m\,dx
 =\int_0^\infty e^{\Psi_S(x)}\,dx,
 \label{eq:transition-J}
\end{equation}
where $x^0=1$ and, for $x>0$,
\[
 \Psi_S(x)=Sx-2\pi e^x+m\log x.
\]
We suppress the dependence of $\Psi_S$, $x_S$, $B_S$, and the probability
measure below on $m$.  Let $x_S$ be the unique maximum of $\Psi_S$:
\begin{equation}
 2\pi e^{x_S}=S+\frac{m}{x_S},
 \label{eq:transition-model-saddle}
\end{equation}
and put
\begin{equation}
 B_S=-\Psi_S''(x_S)=2\pi e^{x_S}+\frac{m}{x_S^2}
 =S+\frac{m}{x_S}+\frac{m}{x_S^2}.
 \label{eq:transition-B}
\end{equation}
Then $B_S\ge S$, $x_S\ge\log(S/(2\pi))$, and $x_{s_0}=x_0$.

\newpage
\begin{lemma}[Uniform saddle estimates]
\label{lem:transition-model-localization}
There are absolute constants $S_1>2\pi$ and $C_1,C_2>0$ with the following
property.  For $S\ge S_1$ and real $m\ge0$, let
\[
 d\varpi_S(x)=\mathcal J_m(S)^{-1}e^{\Psi_S(x)}\,dx\qquad(x>0).
\]
Then the following basic concentration estimates hold:
\begin{align}
 C_1^{-1}e^{\Psi_S(x_S)}B_S^{-1/2}
 &\le \mathcal J_m(S)\le C_1e^{\Psi_S(x_S)}B_S^{-1/2},
 \label{eq:transition-J-size}\\
 \int_0^\infty (x-x_S)^2\,d\varpi_S(x)&\le \frac{C_2}{B_S},
 \label{eq:transition-variance}\\
 \left|\int_0^\infty x\,d\varpi_S(x)-x_S\right|&\le \frac{C_2}{B_S}.
 \label{eq:transition-mean}
\end{align}
For the exponentially weighted tails, fix $\mathfrak a>0$.  There are
$S_{\mathfrak a}\ge S_1$ and constants $c_{\mathfrak a},C_{\mathfrak a}>0$
such that, whenever $S\ge S_{\mathfrak a}$, $H\ge1$, and
$B_S\ge\mathfrak a^{-1}H^2$, one has
\begin{equation}
 \int_{|x-x_S|>1} e^{H|x-x_S|}\,d\varpi_S(x)
 \le C_{\mathfrak a}e^{-c_{\mathfrak a}B_S}.
 \label{eq:transition-weighted-tail}
\end{equation}
Thus the exponential tilt $H$ may reach the natural square-root curvature scale
$B_S^{1/2}$, with constants depending on $\mathfrak a$.  Moreover, for
$1\le r\le H$ and $n\ge1$,
\begin{equation}
 \int_0^\infty e^{-rx}e^{-2\pi(n-1)e^x}\,d\varpi_S(x)
 \le C_{\mathfrak a}^{\,r}e^{-rx_S}e^{-c_{\mathfrak a}(n-1)e^{x_S}}
       +C_{\mathfrak a}e^{-c_{\mathfrak a}B_S}e^{-2\pi(n-1)}.
 \label{eq:transition-weighted-Jn-tail}
\end{equation}
\end{lemma}

\begin{proof}
\noindent\emph{Phase drop and basic concentration.}
Write $u=x-x_S$.  Since $x_S\ge\log(S/(2\pi))$, after increasing $S_1$ we may
assume $x_S\ge4$.  On $|u|\le1$ the two terms $e^{x_S+u}$ and
$(x_S+u)^{-2}$ are comparable with $e^{x_S}$ and $x_S^{-2}$; hence
\[
 cB_S\le -\Psi_S''(x_S+u)\le C B_S\qquad(|u|\le1).
\]
Taylor's formula at the maximum gives
\begin{equation}
 \Psi_S(x_S+u)\le \Psi_S(x_S)-cB_Su^2\qquad(|u|\le1).
 \label{eq:transition-quadratic-drop}
\end{equation}
For $|u|\ge1$ we record the standard concavity step explicitly.  The
function $\Psi_S$ is strictly concave, since
$\Psi_S''(x)=-2\pi e^x-m/x^2<0$.  From
\eqref{eq:transition-quadratic-drop} at $u=1$, the average slope on
$[x_S,x_S+1]$ is at most $-cB_S$; by monotonicity of $\Psi_S'$ this gives
$\Psi_S'(x_S+1)\le -cB_S$, and hence for $u\ge1$,
\[
 \Psi_S(x_S+u)\le \Psi_S(x_S+1)-cB_S(u-1)
              \le \Psi_S(x_S)-cB_Su.
\]
Similarly, the average slope on $[x_S-1,x_S]$ is at least $cB_S$, so
$\Psi_S'(x_S-1)\ge cB_S$.  Since $\Psi_S'$ is decreasing, for $u\le -1$ one has
\[
 \Psi_S(x_S+u)\le \Psi_S(x_S-1)-cB_S(|u|-1)
              \le \Psi_S(x_S)-cB_S|u|.
\]
Together with~\eqref{eq:transition-quadratic-drop}, this gives
\begin{equation}
 \Psi_S(x_S+u)\le \Psi_S(x_S)-cB_S\min(u^2,|u|)
 \label{eq:transition-tail-drop}
\end{equation}
for every $u>-x_S$.

The lower bound in~\eqref{eq:transition-J-size} follows by integrating over
$|u|\le(10\sqrt{B_S})^{-1}$, where Taylor's formula gives
$\Psi_S(x_S+u)\ge\Psi_S(x_S)-C$.  The upper bound in
\eqref{eq:transition-J-size} and the second moment estimate
\eqref{eq:transition-variance} follow by integrating the majorant
\[
 \exp\{-cB_S\min(u^2,|u|)\}.
\]

\smallskip
\noindent\emph{Mean displacement.}
Integration by parts gives
\begin{equation}
 \int_0^\infty \Psi_S'(x)e^{\Psi_S(x)}\,dx
 =\left[e^{\Psi_S(x)}\right]_0^\infty.
 \label{eq:transition-ibp}
\end{equation}
The boundary at infinity is zero.  At the lower endpoint the bracket in
\eqref{eq:transition-ibp} contributes $0$ if $m>0$ and $-e^{-2\pi}$ if
$m=0$.  In the latter case, \eqref{eq:transition-B} gives
$B_S=S$, so \eqref{eq:transition-J-size} and the lower bound
$\Psi_S(x_S)\ge S\log(S/(2\pi))-S$ show that this boundary term is
$O(e^{-cS\log S})\mathcal J_m(S)$.

For $|u|\le1$,
\[
 \Psi_S'(x_S+u)=-B_Su+\mathscr E_S(u),\qquad |\mathscr E_S(u)|\le CB_Su^2.
\]
On $|u|>1$, the two tail integrals can be evaluated from their endpoint
values: since
$(e^{\Psi_S})'=\Psi_S'e^{\Psi_S}$,
\[
\begin{aligned}
 &\int_{0}^{x_S-1}\Psi_S'(x)e^{\Psi_S(x)}dx
 +\int_{x_S+1}^{\infty}\Psi_S'(x)e^{\Psi_S(x)}dx \\
 &\hspace{2em}=e^{\Psi_S(x_S-1)}-\lim_{x\downarrow0}e^{\Psi_S(x)}
  -e^{\Psi_S(x_S+1)}.
\end{aligned}
\]
The two values at $x_S\pm1$ are $O(e^{\Psi_S(x_S)-cB_S})$ by
\eqref{eq:transition-tail-drop}, and the boundary at zero was estimated above.
The same tail majorant, together with~\eqref{eq:transition-J-size}, also gives
\[
 B_S\int_{|x-x_S|>1}|x-x_S|\,d\varpi_S(x)\ll e^{-cB_S}.
\]
Dividing~\eqref{eq:transition-ibp} by $\mathcal J_m(S)$, using the local expansion
of $\Psi_S'$, and then applying~\eqref{eq:transition-variance} and the preceding
tail estimate, after weakening the exponential error we get
\[
 B_S\left|\int(x-x_S)\,d\varpi_S(x)\right|
 \le CB_S\int(x-x_S)^2\,d\varpi_S(x)+O(e^{-cS}),
\]
which proves~\eqref{eq:transition-mean}.

\smallskip
\noindent\emph{Weighted tails.}
Fix $\mathfrak a>0$.  The phase
drop \eqref{eq:transition-tail-drop} and the lower bound in
\eqref{eq:transition-J-size} give
\[
 \int_{|x-x_S|>1}e^{H|x-x_S|}\,d\varpi_S(x)
 \ll B_S^{1/2}\int_1^\infty e^{-(cB_S-H)u}\,du.
\]
The hypothesis $B_S\ge\mathfrak a^{-1}H^2$ implies
$H\le\mathfrak a^{1/2}B_S^{1/2}$.  After increasing $S_{\mathfrak a}$, and
therefore $B_S$, one has $cB_S-H\ge cB_S/2$, so the last display is
$O_{\mathfrak a}(e^{-c_{\mathfrak a}B_S})$.  This proves
\eqref{eq:transition-weighted-tail}.

For \eqref{eq:transition-weighted-Jn-tail}, split the integral into
$|x-x_S|\le1$ and its complement.  On the first region,
$e^{-rx}\le e^r e^{-rx_S}$ and $e^x\ge e^{-1}e^{x_S}$, which gives the first
term after changing the constants.  On the complement,
\[
 e^{-rx}\le e^{-rx_S}e^{r|x-x_S|}\le e^{H|x-x_S|},
 \qquad
 e^{-2\pi(n-1)e^x}\le e^{-2\pi(n-1)}.
\]
Applying \eqref{eq:transition-weighted-tail} gives the second term.
\end{proof}

\begin{lemma}[Uniform model ratio]
\label{lem:transition-model-ratio}
Fix $C>0$.  There exists $c_C>0$ such that, uniformly as $S\to\infty$ for real
$m\ge0$ and real $0\le r\le C\log S$,
\begin{equation}
 \frac{\mathcal J_m(S-r)}{\mathcal J_m(S)}
 =e^{-rx_S}\left(1+O_C\!\left(\frac{(1+r)^2}{S}\right)\right).
 \label{eq:transition-model-ratio}
\end{equation}
Moreover, uniformly for real $1\le S_{\mathrm{lo}}\le C\log S$,
\begin{equation}
 \frac{\mathcal J_m(S_{\mathrm{lo}})}{\mathcal J_m(S)}\le e^{-c_C S}.
 \label{eq:transition-long-drop}
\end{equation}
\end{lemma}

\begin{proof}
For each fixed $m$, set $\mathscr L(S)=\log \mathcal J_m(S)$; all estimates
below are uniform in $m$.  Differentiation under the integral sign on compact
$S$-intervals is justified by the exponential factor $e^{-2\pi e^x}$.  Thus
\[
 \mathscr L'(S)=\int x\,d\varpi_S(x),
 \qquad
 \mathscr L''(S)=\int\bigl(x-\mathscr L'(S)\bigr)^2d\varpi_S(x).
\]
Lemma~\ref{lem:transition-model-localization} gives, uniformly in $m$,
\begin{equation}
 \mathscr L'(S)=x_S+O(S^{-1}),\qquad \mathscr L''(S)=O(S^{-1}).
 \label{eq:transition-L-derivatives}
\end{equation}
The saddle equation implies $dx_S/dS=1/B_S$; since $B_S\ge S$, the saddle moves
by only $O(t/S)$ when $S$ is replaced by $S-t$.  More explicitly, for
$0\le t\le r\le C\log S$ and large $S$,
\[
 x_{S-t}=x_S+O_C(t/S),
 \qquad
 \mathscr L'(S-t)=x_S+O_C((1+t)/S).
\]
Consequently
\[
 \log\frac{\mathcal J_m(S-r)}{\mathcal J_m(S)}
 =-\int_0^r\mathscr L'(S-t)\,dt
 =-rx_S+O_C\!\left(\frac{r+r^2}{S}\right).
\]
Because $r\le C\log S$, the exponential of the final error is
$1+O_C((1+r)^2/S)$, which proves~\eqref{eq:transition-model-ratio}.

For~\eqref{eq:transition-long-drop}, integrate the lower bound for $\mathscr L'$ in
\eqref{eq:transition-L-derivatives}.  For sufficiently large $S$ one has
$S_{\mathrm{lo}}\le C\log S\le S/2$, and $\mathscr L'(u)>0$ because it is the
$\varpi_u$-expectation of $x>0$.  If $u\ge S/2$, then $x_u\ge\log(u/(2\pi))\ge c\log S$, and the
$O(u^{-1})$ error is negligible; hence $\mathscr L'(u)\ge c\log S$ for all sufficiently
large $S$, uniformly in $m$.  Thus
\[
 \mathscr L(S)-\mathscr L(S_{\mathrm{lo}})\ge\int_{S/2}^S \mathscr L'(u)\,du\ge cS\log S,
\]
which is stronger than~\eqref{eq:transition-long-drop}.
\end{proof}

\subsection{From the model integral to modular \texorpdfstring{$L$}{L}-derivatives}

The split-Mellin formula separates the right-edge contribution from its
functional-equation partner.  We record the exact identity before estimating
either term, since it will also be used in the fixed-weight argument.  Deligne's
bound then makes the $n=1$ term dominant on the right edge, while the reflected
Mellin term is exponentially smaller in the range $r=O(\log k)$.

\begin{lemma}[Split-Mellin identity]
\label{lem:transition-split-mellin}
For every real $s$ and every integer $m\ge0$,
\begin{equation}
 \Lambda_f^{(m)}(s)=\mathcal I_{f,m}(s)+(-1)^{k/2+m}\mathcal I_{f,m}(k-s).
 \label{eq:transition-split-mellin}
\end{equation}
\end{lemma}

\begin{proof}
Split the Mellin integral at $y=1$.  On $[1,\infty)$, differentiation under the
integral sign is justified by cusp decay.  In the lower integral, first apply
$y\mapsto1/y$ and the modular transformation
$f(i/y)=(-1)^{k/2}y^k f(iy)$, thereby reducing it to another integral over
$[1,\infty)$; differentiation there is justified in the same way.  Differentiating
$m$ times then gives the stated identity, with the second sign coming from
$\log(1/y)^m=(-\log y)^m$.
\end{proof}

Absolute convergence of the Fourier expansion on $y\ge1$ gives, for every real
$S$,
\begin{equation}
 \mathcal I_{f,m}(S)=\sum_{n\ge1}a_f(n)\mathcal J_{n,m}(S).
 \label{eq:transition-I-Fourier}
\end{equation}

\begin{proof}[Proof of Proposition~\ref{prop:transition-near-edge-ratio}]
By Lemma~\ref{lem:transition-split-mellin}, when $s=s_0-r$ the two Mellin
parameters are $s_0-r$ and $1+r$.

For $\mathcal J_{n,m}$ as in \eqref{eq:model-moments}, the change of variables
$t=ny$ and the inequality $\log(t/n)\le\log t$ for
$t\ge n\ge1$ give
\begin{equation}
 \mathcal J_{n,m}(S)\le n^{-S}\mathcal J_m(S).
 \label{eq:transition-nth-bound}
\end{equation}
For $S=s_0-r=k-1-r$, Deligne's bound~\cite{Deligne}, in the form
$|a_f(n)|\le \tau(n)n^{(k-1)/2}$ with $\tau$ the divisor function, implies
\begin{align*}
 \sum_{n\ge2}|a_f(n)|\mathcal J_{n,m}(S)
 &\le \mathcal J_m(S)
 \sum_{n\ge2}\tau(n)n^{r-(k-1)/2}
 =O_C(e^{-c_Ck})\mathcal J_m(S).
\end{align*}
Together with \eqref{eq:transition-I-Fourier}, this gives
\begin{equation}
 \mathcal I_{f,m}(s_0-r)=\mathcal J_m(s_0-r)(1+O_C(e^{-c_Ck})).
 \label{eq:transition-main-split}
\end{equation}
The $n=1$ contribution to the reflected Mellin term is
$\mathcal J_m(1+r)$.  By \eqref{eq:transition-long-drop}, it is
$O_C(e^{-c_Ck})\mathcal J_m(s_0-r)$.  For $n\ge2$, the condition
$r\le C\log k$ gives $1+r\le s_0-r$ for all large $k$, and hence, directly
from the defining integral,
\[
 \mathcal J_{n,m}(1+r)\le \mathcal J_{n,m}(s_0-r).
\]
Applying~\eqref{eq:transition-nth-bound} at $S=s_0-r$ and Deligne's bound
therefore gives
\begin{align*}
 \sum_{n\ge2}|a_f(n)|\mathcal J_{n,m}(1+r)
 &\le \mathcal J_m(s_0-r)
      \sum_{n\ge2}\tau(n)n^{r-(k-1)/2}\\
 &\le O_C(e^{-c_Ck})\mathcal J_m(s_0-r).
\end{align*}
Therefore
\begin{equation}
 \Lambda_f^{(m)}(s_0-r)
 =\mathcal J_m(s_0-r)(1+O_C(e^{-c_Ck}))
 \label{eq:transition-Lambda-to-J}
\end{equation}
for every $0\le r\le C\log k$, uniformly in $m$ and $f$.

Applying~\eqref{eq:transition-Lambda-to-J} with $r$ and with $r=0$, and then
using Lemma~\ref{lem:transition-model-ratio}, gives
\begin{equation}
 \frac{\Lambda_f^{(m)}(s_0-r)}{\Lambda_f^{(m)}(s_0)}
 =e^{-rx_0}\left(1+O_C\!\left(\frac{(1+r)^2}{k}\right)\right).
 \label{eq:transition-Lambda-ratio}
\end{equation}
The gamma normalization satisfies
\begin{equation}
 \frac{G(s_0)}{G(s_0-r)}
 =(2\pi)^{-r}\frac{\Gamma(s_0)}{\Gamma(s_0-r)}
 =\left(\frac{s_0}{2\pi}\right)^r
 \left(1+O_C\!\left(\frac{r^2}{k}\right)\right),
 \label{eq:transition-G-ratio}
\end{equation}
since $\Gamma(s_0)/\Gamma(s_0-r)=\prod_{j=1}^r(s_0-j)$ and
$\sum_{j\le r}\log(1-j/s_0)=O_C(r^2/k)$.  Multiplying
\eqref{eq:transition-Lambda-ratio} and~\eqref{eq:transition-G-ratio} proves
\eqref{eq:transition-ratio-proved}.

\end{proof}

\subsection{Factorial tails and the moving-edge approximation}

The ratio estimate just proved controls only logarithmically many coefficients.
The next recurrence supplies a uniform factorial envelope for all edge
coefficients; combining the envelope with the local ratios gives the compact
moving-edge approximation.

\begin{lemma}[A backward Mellin recurrence]
\label{lem:transition-backward-recurrence}
For $n\ge1$, integer $m\ge0$, and $S>1$, let $\mathcal J_{n,m}(S)$ be as in
\eqref{eq:model-moments}.  Then
\begin{equation}
 \mathcal J_{n,m}(S-1)\le \frac{2\pi n}{S-1}\mathcal J_{n,m}(S).
 \label{eq:transition-backward-one-step}
\end{equation}
Consequently, if $S$ is a positive integer and $0\le r\le S-1$, then
\begin{equation}
 \mathcal J_{n,m}(S-r)
 \le (2\pi n)^r\frac{(S-r-1)!}{(S-1)!}\mathcal J_{n,m}(S).
 \label{eq:transition-backward-iterated}
\end{equation}
\end{lemma}

\begin{proof}
Integrating the derivative of $e^{-2\pi n y}y^{S-1}(\log y)^m$ on
$[1,\infty)$ gives, when $m=0$,
\[
 2\pi n \mathcal J_{n,0}(S)
 =(S-1)\mathcal J_{n,0}(S-1)+e^{-2\pi n},
\]
and, when $m\ge1$,
\[
 2\pi n \mathcal J_{n,m}(S)
 =(S-1)\mathcal J_{n,m}(S-1)
  +m\mathcal J_{n,m-1}(S-1).
\]
All additional terms on the right are nonnegative, so
\eqref{eq:transition-backward-one-step} follows in both cases.  Iteration gives
\eqref{eq:transition-backward-iterated}.
\end{proof}

\begin{proposition}[Uniform global factorial tail]
\label{prop:transition-global-factorial-tail}
Let $d=k-2=2M$, $s_0=k-1$, and define
\[
 c_r=\mathcal D_{f,m}(s_0-r)\frac{(2\pi)^r}{r!}\qquad(0\le r\le M)
\]
as in the right-edge decomposition.  Let
\[
 \widetilde c_r=c_r\quad(0\le r<M),\qquad
 \widetilde c_M=\frac12c_M,
\]
so that $\widetilde c_r$ are exactly the coefficients of the edge polynomial
$q_{f,m,k}$.  For every fixed $R>0$ and every $A>0$, there are positive constants
$C_R$, $C_{R,A}$, and $K_*=K_*(R,A)$ such that, for all even weights
$k\ge K_*$, all integers $m\ge0$, all normalized Hecke eigenforms $f\in S_k(\SLZ)$,
and all integers $L$ with $0\le L\le M$, one has $c_0\ne0$ and
\begin{equation}
 \sup_{|w|\le R}\sum_{r=L}^{M}\left|\frac{\widetilde c_r}{c_0}w^r\right|
 \le C_R\sum_{r=L}^{\infty}\frac{(2\pi R)^r}{r!}+C_{R,A}k^{-A}.
 \label{eq:transition-global-factorial-tail}
\end{equation}
The factorial envelope is independent of $A$; only the super-polynomial
remainder and its constant depend on $A$.  The same estimate also holds with
$\widetilde c_r$ replaced by $c_r$, after multiplying the right-hand side by
at most $2$.  The constants are independent of $L$, $m$, and the normalized Hecke
eigenform $f\in S_k(\SLZ)$.
\end{proposition}

\begin{proof}
Write $\kappa=(k-1)/2=M+1/2$ and normalize the Fourier coefficients by
\[
 a_f(n)=\rho_f(n)n^\kappa.
\]
Deligne's bound gives $|\rho_f(n)|\le\tau(n)$, where $\tau$ denotes the
divisor function~\cite{Deligne}.  By \eqref{eq:model-moments},
$\mathcal J_m=\mathcal J_{1,m}$, and the split-Mellin transform has the Fourier
expansion \eqref{eq:transition-I-Fourier}.  The split-Mellin formula is
\begin{equation}
 \Lambda_f^{(m)}(s_0-r)=\mathcal I_{f,m}(s_0-r)+(-1)^{k/2+m}\mathcal I_{f,m}(1+r).
 \label{eq:transition-tail-split}
\end{equation}

We first isolate the denominator.  The change of variables $t=ny$ gives, for
$S\ge1$,
\[
 \mathcal J_{n,m}(S)
 =n^{-S}\int_n^\infty e^{-2\pi t}t^S(\log(t/n))^m\,\frac{dt}{t}
 \le n^{-S}\mathcal J_m(S).
\]
Hence
\[
 \mathcal I_{f,m}(s_0)=\mathcal J_m(s_0)+O\!\left(\mathcal J_m(s_0)
        \sum_{n\ge2}\tau(n)n^{\kappa-s_0}\right)
        =\mathcal J_m(s_0)(1+O(2^{-M})).
\]
The reflected Mellin term in $\Lambda_f^{(m)}(s_0)$ is $\mathcal I_{f,m}(1)$, and it is covered
by the following simple bound, which we shall also need below.  If $1\le B\le M+1$,
then $\mathcal J_{n,m}(B)\le \mathcal J_{n,m}(M+1)$; after the same change of variables,
\begin{align*}
 |\mathcal I_{f,m}(B)|
 &\le \sum_{n\ge1}\tau(n)n^\kappa\mathcal J_{n,m}(M+1) \\
 &\le \int_1^\infty e^{-2\pi t}t^{M+1}
  \sum_{n\le t}\tau(n)n^{-1/2}\bigl(\log(t/n)\bigr)^m\,\frac{dt}{t}.
\end{align*}
Using the elementary divisor estimate $\tau(n)\le2n^{1/2}$ and
$\log(t/n)\le\log t$, we obtain
\begin{equation}
 |\mathcal I_{f,m}(B)|\le 2\mathcal J_m(M+2)\qquad(1\le B\le M+1).
 \label{eq:transition-reflected-bound}
\end{equation}
By Lemma~\ref{lem:transition-backward-recurrence},
\begin{equation}
 \frac{\mathcal J_m(M+2)}{\mathcal J_m(2M+1)}
 \le (2\pi)^{M-1}\frac{(M+1)!}{(2M)!}=O_{A}(k^{-A})
 \label{eq:transition-reflected-small}
\end{equation}
for every fixed $A>0$, uniformly in $m$.  Therefore
\begin{equation}
 \Lambda_f^{(m)}(s_0)=\mathcal J_m(s_0)(1+O_{A}(k^{-A})),
 \qquad |\Lambda_f^{(m)}(s_0)|\ge \tfrac12 \mathcal J_m(s_0)
 \label{eq:transition-denominator-lower}
\end{equation}
for all sufficiently large $k$, uniformly in $m$ and $f$.  In particular,
the normalization by $c_0$ used in all subsequent large-weight arguments is
valid simultaneously for every derivative order and every normalized
level-one eigenform.

\noindent\emph{Right-edge main terms, $0\le r\le M-1$.}
The factorial normalization is exactly the one produced by the backward
recurrence, since
\[
 \binom dr(2\pi n)^r\frac{(s_0-r-1)!}{(s_0-1)!}
 =\frac{(2\pi n)^r}{r!}
 \qquad(d=s_0-1).
\]
Thus the backward recurrence and the preceding scaling bound give
\begin{align}
 \binom dr\frac{|\mathcal I_{f,m}(s_0-r)|}{\mathcal J_m(s_0)}
 &\le \binom dr\sum_{n\ge1}\tau(n)n^\kappa
       \frac{\mathcal J_{n,m}(s_0-r)}{\mathcal J_m(s_0)} \notag\\
 &\le \frac{(2\pi)^r}{r!}
       \sum_{n\ge1}\tau(n)n^{\kappa+r-s_0}
 \ll \frac{(2\pi)^r}{r!}.
 \label{eq:transition-main-factorial-bound}
\end{align}
The last sum is uniformly bounded because
$\kappa+r-s_0=r-M-1/2\le-3/2$ in this range.

\smallskip
\noindent\emph{The central coefficient, $r=M$.}
Here both Mellin parameters in \eqref{eq:transition-tail-split} are $M+1$.
By \eqref{eq:transition-reflected-bound},
\[
 \binom{2M}{M}\frac{|\mathcal I_{f,m}(M+1)|}{\mathcal J_m(2M+1)}
 \le 2\binom{2M}{M}\frac{\mathcal J_m(M+2)}{\mathcal J_m(2M+1)}.
\]
Moreover,
$|\Lambda_f^{(m)}(M+1)|\le2|\mathcal I_{f,m}(M+1)|$.  Using
\eqref{eq:transition-denominator-lower}, it follows that the normalized full
coefficient satisfies
\[
 \left|\frac{c_M}{c_0}\right|=O_A(k^{-A})
\]
for every fixed $A>0$, uniformly in $m$ and $f$.  The actual coefficient of
$w^M$ in $q_{f,m,k}$ is $\widetilde c_M=c_M/2$, which is smaller.

\smallskip
\noindent\emph{Reflected Mellin terms, $0\le r\le M$.}
These are also controlled by \eqref{eq:transition-reflected-bound}.  Since
$\binom{2M}{r}\le\binom{2M}{M}$,
\begin{align}
 \binom dr\frac{|\mathcal I_{f,m}(1+r)|}{\mathcal J_m(s_0)}
 &\le 2\binom{2M}{M}\frac{\mathcal J_m(M+2)}{\mathcal J_m(2M+1)} \notag\\
 &\ll (2\pi)^{M-1}\frac{M+1}{M!},
 \label{eq:transition-reflected-factorial-bound}
\end{align}
for every $0\le r\le M$.  After multiplication by $R^r$ and summation,
\[
 \sum_{r=0}^{M}R^r\,(2\pi)^{M-1}\frac{M+1}{M!}
 \le (1+R)^M(2\pi)^{M-1}\frac{M+1}{M!}
 =O_{R,A}(k^{-A})
\]
for every fixed $R$ and $A>0$, by Stirling's formula.  Thus all terms not
covered by the factorial main-term estimate---namely the reflected Mellin terms
and the central coefficient---contribute $O_{R,A}(k^{-A})$, uniformly in the
lower summation cutoff $L$.

Finally, for $0\le r\le M$,
\[
 \frac{c_r}{c_0}w^r
 =\frac{\Lambda_f^{(m)}(s_0-r)}{\Lambda_f^{(m)}(s_0)}
  \binom dr w^r.
\]
The displayed identity is used with $c_M$; replacing $c_M$ by
$\widetilde c_M=c_M/2$ only decreases the left hand side of the proposition.
Combining
\eqref{eq:transition-tail-split}, \eqref{eq:transition-denominator-lower},
\eqref{eq:transition-main-factorial-bound}, and
\eqref{eq:transition-reflected-factorial-bound}, and summing the main-term
majorant only over $r\ge L$, gives \eqref{eq:transition-global-factorial-tail}
with constants independent of $L$.
\end{proof}

\begin{proposition}[Quantitative moving-edge approximation]
\label{prop:transition-moving-edge-quantitative}
For every fixed $R>0$, there exists $K_R$ such that, for every even
$k\ge K_R$, every integer $m\ge0$, and every normalized level-one Hecke
eigenform $f\in S_k(\SLZ)$, one has $c_0\ne0$, and
\begin{equation}
 \sup_{|w|\le R}
 \left|\frac{q_{f,m,k}(w)}{c_0}-e^{2\pi\alpha_{k,m}w}\right|
 \ll_R \frac1k.
 \label{eq:moving-edge-quantitative}
\end{equation}
Consequently,
\begin{equation}
 \sup_{|z|\le R}
 \left|p_{f,m,k}(z)-\sin(2\pi\alpha_{k,m}z)\right|
 \ll_R \frac1k.
 \label{eq:moving-p-quantitative}
\end{equation}
Both estimates are uniform in integers $m\ge0$ and in the normalized level-one
Hecke eigenform $f$.
\end{proposition}

\begin{proof}
Choose $B=B(R)$ so large that, with $L=\lfloor B\log k\rfloor$, the tail of the
exponential series over $r\ge L+1$ is $O_R(k^{-2})$ uniformly for $|w|\le R$ and
$0<\alpha_{k,m}\le1$, as recorded in \eqref{eq:transition-alpha}.
For all sufficiently large $k$ one has $L+1\le M$.  Proposition~\ref{prop:transition-global-factorial-tail}, applied with $A=2$ and lower cutoff $L+1$,
gives the same bound, up to $O_R(k^{-2})$, for the tail of $q_{f,m,k}/c_0$.  For $0\le r\le L$,
Proposition~\ref{prop:transition-near-edge-ratio} gives
\[
 \frac{c_r}{c_0}
 =\left(\alpha_{k,m}\right)^r\frac{(2\pi)^r}{r!}
 \left(1+O_R\left(\frac{(1+r)^2}{k}\right)\right).
\]
Summing these coefficient errors with the exponential weights gives
\[
 \sum_{0\le r\le L}O_R\left(\frac{(1+r)^2}{k}\right)\frac{(2\pi R)^r}{r!}=O_R(k^{-1}),
\]
because the full weighted second moment of the exponential series is bounded
in terms of $R$ alone.  This proves \eqref{eq:moving-edge-quantitative}.
Applying the definition \eqref{eq:odd-edge-p} at $w=iz$ and $w=-iz$ gives
\eqref{eq:moving-p-quantitative}.
\end{proof}

\section{Separated and critical zero phases}
\label{sec:moving-phases}

The quantitative edge estimate reduces the zero problem to the model
$\sin(2\pi\alpha z)$.  The boundary-margin estimate singles out
$\alpha=0$, $1/2$, and $1$ as the only degenerate values.  This section treats
the separated ranges and the collision at $\alpha=1/2$; Sections~\ref{sec:lower-endpoint}
and~\ref{sec:upper-endpoint} treat the two endpoints.

\subsection{The phase parameter and the critical order}

\begin{lemma}[Monotonicity and the critical order]
\label{lem:alpha-monotone-m}
Let $s=k-1$ and assume $s>2\pi$.  For real $m\ge0$, let $x_0=x_0(k,m)$ be the
positive solution of
\[
 2\pi e^{x_0}=s+\frac{m}{x_0},
\]
and set
\[
 \alpha_{k,m}=\frac{s}{2\pi e^{x_0}}.
\]
Then $m\mapsto\alpha_{k,m}$ is strictly decreasing.  Moreover, with
\[
 m_c(k):=s\log\frac{s}{\pi}=(k-1)\log\frac{k-1}{\pi},
\]
one has
\[
\begin{aligned}
 \alpha_{k,m}>\frac12 &\iff m<m_c(k),\\
 \alpha_{k,m}=\frac12 &\iff m=m_c(k),\\
 \alpha_{k,m}<\frac12 &\iff m>m_c(k).
\end{aligned}
\]
\end{lemma}

\begin{proof}
Parametrize the saddle equation by $x$ rather than by $m$:
\[
 m=x(2\pi e^x-s),\qquad x\ge \log\frac{s}{2\pi}.
\]
Since $s>2\pi$, this interval lies in $x>0$, and
\[
 \frac{d}{dx}\{x(2\pi e^x-s)\}=2\pi e^x(1+x)-s\ge sx>0
\]
throughout this interval.  Thus $x_0$ is strictly increasing with $m$, and
\[
 \alpha_{k,m}=\frac{s}{2\pi e^{x_0}}
\]
is strictly decreasing.  Finally, $\alpha_{k,m}=1/2$ is equivalent to
$2\pi e^{x_0}=2s$, hence $x_0=\log(s/\pi)$; substituting into
$m=x_0(2\pi e^{x_0}-s)$ gives $m=s\log(s/\pi)=m_c(k)$.  Monotonicity gives the
two strict inequalities.
\end{proof}

\begin{lemma}[Width of the critical collar]
\label{lem:critical-collar-width}
Let $k\to\infty$ through even weights, let $m=m(k)\in\Z_{\ge0}$,
put $s=k-1$, and set
\[
 x_c=\log\frac{s}{\pi},\qquad m_c(k)=s x_c,
 \qquad \Delta_k=m-m_c(k).
\]
If $\Delta_k/\log k\to+\infty$, then
\[
 \alpha_{k,m}<\frac12\quad\text{eventually},
 \qquad k\left(\frac12-\alpha_{k,m}\right)\to+\infty.
\]
If $-\Delta_k/\log k\to+\infty$, then
\[
 \alpha_{k,m}>\frac12\quad\text{eventually},
 \qquad k\left(\alpha_{k,m}-\frac12\right)\to+\infty.
\]
More precisely, in the near-critical subrange $|\Delta_k|=o(k\log k)$ one has
\begin{equation}
 \alpha_{k,m}-\frac12
 =-\frac{\Delta_k}{2s(1+2x_c)}
  +O\left(\frac{\Delta_k^2}{s^2x_c^2}\right).
 \label{eq:alpha-critical-linearization}
\end{equation}
\end{lemma}

The proof is deferred to Section~\ref{sec:derivative-axis-consequences}, where
the lemma is used to resolve the logarithmic collision window.

\subsection{Separated phases away from the collision}

Whenever the sine boundary margin dominates the $O(k^{-1})$ edge error,
ordinary Rouch\'e comparison determines the disk zeros of the right-edge odd
polynomial.  We first treat the two phases separated from the collision; the
endpoint losses are isolated in Sections~\ref{sec:lower-endpoint}
and~\ref{sec:upper-endpoint}.

\begin{lemma}[Sine boundary margin]
\label{lem:sine-boundary-margin}
There are absolute constants $c,C>0$ such that, for every $0<\alpha<1$,
\[
 c\min\{\alpha,\ |2\alpha-1|,\ 1-\alpha\}
 \le \min_{|z|=1}|\sin(2\pi\alpha z)|
 \le C\min\{\alpha,\ |2\alpha-1|,\ 1-\alpha\}.
\]
\end{lemma}

\begin{proof}
On the compact disk $|w|\le2\pi$, the zeros of $\sin w$ are simple.  Hence
\[
 \frac{|\sin w|}{\operatorname{dist}(w,\pi\Z)}
\]
extends across those zeros to a positive continuous function and is bounded
above and below by positive absolute constants.  Here
$\operatorname{dist}$ denotes Euclidean distance.  If $|z|=1$, then
$w=2\pi\alpha z$ runs over the circle of radius $2\pi\alpha$.  Since
$0<2\pi\alpha<2\pi$, the only relevant radii of points in $\pi\Z$ are
$0$, $\pi$, and $2\pi$, and therefore
\[
 \min_{|z|=1}\operatorname{dist}(2\pi\alpha z,\pi\Z)
 =\pi\min\{2\alpha,\ |2\alpha-1|,\ 2(1-\alpha)\}.
\]
This quantity is comparable, with absolute constants, to
$\min\{\alpha,|2\alpha-1|,1-\alpha\}$, which proves both bounds.
\end{proof}

\begin{remark}[Degenerate regimes of the sine comparison]
\label{rem:organizing-boundary-margin}
Combining Proposition~\ref{prop:transition-moving-edge-quantitative} with
Lemma~\ref{lem:sine-boundary-margin} gives the ordinary Rouch\'e argument
whenever
\begin{equation}
 k\min\{\alpha_{k,m},\ |2\alpha_{k,m}-1|,\ 1-\alpha_{k,m}\}
 \longrightarrow\infty.
 \label{eq:zeroth-order-rouche-criterion}
\end{equation}
Failure of this condition can occur only as $\alpha_{k,m}\to0$, $1/2$, or
$1$.  At the lower endpoint, a relative weighted estimate is needed because
the sine model itself is small.  At the critical value, a boundary-sensitive
count is needed because model zeros meet $\T$.  At the upper endpoint, a
first-order expansion is needed because the relevant displacement is no larger
than the absolute error.  The critical regime is treated below; the two
endpoints are treated in Sections~\ref{sec:lower-endpoint}
and~\ref{sec:upper-endpoint}.
\end{remark}

\newpage
\begin{theorem}[Separated moving phase]
\label{thm:separated-moving-phase}
Let $k\to\infty$ through even weights and let $m=m(k)\in\Z_{\ge0}$.  Put
$\alpha_k:=\alpha_{k,m(k)}$ and $E_k:=k^{-1}$.  The following assertions hold uniformly over
normalized level-one Hecke eigenforms.

If $0<\alpha_k<1/2$ and
\begin{equation}
 E_k=o\!\left(\min\{\alpha_k,\,1/2-\alpha_k\}\right),
 \label{eq:separated-lower-phase-condition}
\end{equation}
then, for all sufficiently large $k$, the origin is a simple zero of
$Q_{f,m}^{-}$ and every nonzero zero of $Q_{f,m}^{-}$ is simple and lies on
$\T$.

If $1/2<\alpha_k<1$ and
\begin{equation}
 E_k=o\!\left(\min\{\alpha_k-1/2,\,1-\alpha_k\}\right),
 \label{eq:separated-upper-phase-condition}
\end{equation}
then, for all sufficiently large $k$, $Q_{f,m}^{-}$ has a simple zero at $0$ and
four simple zeros
\[
 \pm b_{f,m,k},\qquad \pm b_{f,m,k}^{-1},
\]
forming a single real reciprocal quartet, where $0<b_{f,m,k}<1$, and
\begin{equation}
 b_{f,m,k}=\frac1{2\alpha_k}+O(E_k),
 \qquad
 b_{f,m,k}^{-1}=2\alpha_k+O(E_k),
 \label{eq:separated-quartet-location}
\end{equation}
and all remaining zeros are simple and lie on $\T$.
\end{theorem}

\begin{proof}
Put $N=2M$.  By Proposition~\ref{prop:transition-moving-edge-quantitative},
$c_0\ne0$ for all sufficiently large $k$.  Hence the polynomials
$p_{f,m,k}$ and $P_{f,m,k}$ in
\eqref{eq:odd-edge-p}--\eqref{eq:normalized-odd-completion} are defined; write
$p=p_{f,m,k}$ and $P=P_{f,m,k}$ for the duration of the proof.  The sine
estimate \eqref{eq:moving-p-quantitative} from the same proposition holds for
each fixed $R$, and in particular for $R=3$ below, with
$\alpha_{k,m}=\alpha_k$ and $k^{-1}=E_k$.
Lemma~\ref{lem:sine-boundary-margin} gives the required lower bound for the sine
model on $\T$ in both separated ranges.

\smallskip
\noindent\emph{The lower separated phase.}
Under~\eqref{eq:separated-lower-phase-condition}, Rouch\'e's theorem on $\T$
therefore shows that $p$ has exactly the same number of zeros in $\D$ as
$\sin(2\pi\alpha_k z)$, namely the simple zero at $0$, and no zero on $\T$.
Cauchy's estimate applied to~\eqref{eq:moving-p-quantitative} on a fixed
circle around the origin gives
$p'(0)=2\pi\alpha_k+O(E_k)\ne0$.  The completion $P$ is,
by Lemma~\ref{lem:odd-part-completion}, an odd polynomial of degree $N-1$
with a simple zero at the origin.  It then satisfies
Lemma~\ref{lem:disk-zero-circle-count} with $r=1$.  The origin together with
the at least $N-2$ distinct unit-circle zeros accounts for the full degree.
Hence every nonzero zero lies on $\T$, and all nonzero zeros are simple.

\smallskip
\noindent\emph{The upper separated phase.}
Under~\eqref{eq:separated-upper-phase-condition}, Rouch\'e's theorem on $\T$
shows that $p$ has exactly three zeros in $\D$ and none on $\T$.  Let
$\beta_k=(2\alpha_k)^{-1}$.  Since $E_k=o(\alpha_k-1/2)$, disjoint conjugation-invariant
disks of radius $C E_k$ around $\pm\beta_k$ are contained in $\D$ and stay a
distance $\asymp\alpha_k-1/2$ from $\T$.  On their boundaries the model sine
has size $\gg C E_k$, while the perturbation is $O(E_k)$ and is dominated after
choosing $C$ large enough.
Thus $p$ has one simple zero in each disk.  Since the disks are
invariant under conjugation and $p$ has real coefficients, these zeros
are real.  Cauchy's estimate applied to~\eqref{eq:moving-p-quantitative} also
gives
$p'(0)=2\pi\alpha_k+O(E_k)\ne0$, so
Lemma~\ref{lem:odd-part-completion} shows that the completion $P$ has degree
$N-1$ and a simple zero at the origin.

Moreover
\[
 1-\beta_k=\frac{2\alpha_k-1}{2\alpha_k}\asymp \alpha_k-\frac12,
\]
and $E_k=1/k=o(\alpha_k-1/2)$ implies
$N(\alpha_k-1/2)\to\infty$.  It remains to control the reflected term in the
subrange where the inner zero approaches the unit circle.  In this subrange the
single estimate $p(1/z)=O(\alpha_k-1/2)$ is not uniform over the whole range
$1/2<\alpha_k<1$.

Fix a small absolute $\gamma>0$.  If $\alpha_k-1/2\ge\gamma$, then the disks around
$\pm\beta_k$ are contained in $|z|\le\varrho_\gamma<1$ for all large $k$.  By
\eqref{eq:moving-p-quantitative} with $R=3$, $p(1/z)=O_\gamma(1)$ on
those disks, and hence
\[
 |z|^N|p(1/z)|\ll_\gamma \varrho_\gamma^N=o(E_k).
\]

It remains to consider the near-collision subrange
$0<\alpha_k-1/2<\gamma$.  For $z$ in either disk,
$1/z=\pm2\alpha_k+O(E_k)$, and hence, uniformly on the two disks,
\[
 \sin\!\left(2\pi\alpha_k\frac1z\right)
 =\sin(\pm4\pi\alpha_k^2)+O(E_k)
 =O(\alpha_k-1/2)+O(E_k).
\]
The moving-edge approximation therefore gives
\[
 p(1/z)=O(\alpha_k-1/2)+O(E_k)=O(\alpha_k-1/2).
\]
Also $|z|\le\exp\{-c(\alpha_k-1/2)\}$ on the two disks, after increasing $k$ if
necessary.  Therefore, with $x_k=N(\alpha_k-1/2)\to\infty$,
\[
 |z|^N|p(1/z)|
 \ll (\alpha_k-1/2)e^{-cx_k}
 =O(x_ke^{-cx_k}/N)=o(E_k).
\]
In both subranges the reflected term $z^Np(1/z)$ is smaller than the
Rouch\'e margin on the two small boundaries.  Rouch\'e's theorem gives one zero
of the completion in each of the two conjugation-invariant disks.  Since the completion
has real coefficients, the unique zero in each disk is real; the disk radii
are $O(E_k)$, so the two inner zeros have locations
$\pm((2\alpha_k)^{-1}+O(E_k))$.  Signed reciprocity supplies the reciprocal
pair.  Since
\[
 \bigl((2\alpha_k)^{-1}+O(E_k)\bigr)^{-1}
 =2\alpha_k+O(E_k),
\]
this gives the second formula in \eqref{eq:separated-quartet-location}.  For all sufficiently large $k$, one has
$3\le N/2$, so Lemma~\ref{lem:disk-zero-circle-count}, applied with $r=3$,
gives at least $N-6$ distinct unit-circle zeros.  These distinct
unit-circle zeros, the simple zero at the origin, and the four zeros in the real
reciprocal quartet account for the full degree $N-1$.  Thus all remaining zeros
lie on $\T$ and are simple.
\end{proof}

\subsection{The critical collision}

At the formal critical value $m_c(k)$, and more generally whenever
$\alpha_{k,m}\to1/2$, the moving model develops boundary zeros at $\pm1$.
The critical containment is proved by combining the annular exclusion below
with the boundary-sensitive winding count of
Lemma~\ref{lem:disk-zero-circle-count-boundary}, reality, oddness, and signed
reciprocity.  The winding count leaves room for at most four off-circle zeros,
whereas localization near $\pm1$ makes the orbit of any nonreal off-circle zero
under conjugation, sign change, and inversion consist of eight distinct points.

\newpage
\begin{lemma}[Annular exclusion away from the critical collars]
\label{lem:critical-annular-exclusion}
Let $N_j\in\Z_{\ge1}$ with $N_j\to\infty$, let
$\sigma_j\in\{\pm1\}$, and let $p_j\in\R[z]$ be polynomials with
$\deg p_j<N_j$ that converge locally uniformly to
$\sin(\pi z)$.  Put
\[
 P_j(z)=p_j(z)+\sigma_j z^{N_j}p_j(1/z),
\]
so that $P_j$ is an ordinary polynomial.  Assume that $P_j\not\equiv0$ for all
$j$.
Fix $0<u_{\rm ann}<u_{\rm col}<\log(3/2)$ and $0<\vartheta_{\rm col}<\pi/4$, and define
\[
 \mathscr A_{\rm ann}=\{e^{-u_{\rm ann}}\le |z|\le e^{u_{\rm ann}}\},
\]
\[
 U_+=\{e^{u+i\vartheta}: |u|<u_{\rm col},\ |\vartheta|<\vartheta_{\rm col}\},
\]
\[
 U_-=\{-e^{u+i\vartheta}: |u|<u_{\rm col},\ |\vartheta|<\vartheta_{\rm col}\}.
\]
Then, for all sufficiently large $j$, every
zero of $P_j$ in $\mathscr A_{\rm ann}\setminus(U_+\cup U_-)$ lies on $\T$.
\end{lemma}

\begin{proof}
Because $u_{\rm col}>u_{\rm ann}$, the set
$\mathscr A_{\rm ann}\setminus(U_+\cup U_-)$ is exactly the image of the compact
log-polar set
\[
 \mathscr K=
 \left\{(u,\vartheta): |u|\le u_{\rm ann},\
 \operatorname{dist}_{\R/2\pi\Z}(\vartheta,\{0,\pi\})
 \ge \vartheta_{\rm col}\right\},
\]
where $\operatorname{dist}_{\R/2\pi\Z}$ denotes angular distance modulo
$2\pi$.  Choose a slightly larger compact log-polar set $\mathscr K^\ast$ by replacing
$u_{\rm ann}$ with $u_{\rm col}$ and $\vartheta_{\rm col}$ with
$\vartheta_{\rm col}/2$.  Since $u_{\rm col}<\log(3/2)$, its image
$K^\ast$ lies in $\{2/3<|z|<3/2\}$ and avoids the only zeros $\pm1$ of
$\sin(\pi z)$ in that annulus.  Since $\sin(\pi z)$ is zero-free on
$K^\ast$, local uniform convergence gives $|p_j(z)|\ge c>0$ there for all
sufficiently large $j$.  Cauchy's formula on a fixed neighborhood of
$K^\ast$ gives $|p_j'(z)|\le C$.  Hence, for some constant $D>0$,
\begin{equation}
 \left|\frac{p_j'(z)}{p_j(z)}\right|\le D
 \qquad(z\in K^\ast)
 \label{eq:critical-log-derivative-bound}
\end{equation}
for all sufficiently large $j$.

Let $z=re^{i\vartheta}$ lie in
$\mathscr A_{\rm ann}\setminus(U_+\cup U_-)$.  The entire radial segment
$e^{v+i\vartheta}$ with $|v|\le |\log r|$ remains in $K^\ast$.  Hence
\[
 \frac{d}{dv}\log|p_j(e^{v+i\vartheta})|
 =\Re\left(e^{v+i\vartheta}
 \frac{p_j'(e^{v+i\vartheta})}{p_j(e^{v+i\vartheta})}\right),
\]
and \eqref{eq:critical-log-derivative-bound} gives
\begin{equation}
 \left|\log\left|
 \frac{p_j(re^{i\vartheta})}{p_j(r^{-1}e^{i\vartheta})}
 \right|\right|
 \le D'|\log r|,
 \qquad D':=2e^{u_{\rm col}}D.
 \label{eq:critical-radial-ratio}
\end{equation}

Both $p_j(re^{i\vartheta})$ and $p_j(r^{-1}e^{-i\vartheta})$ are nonzero,
because these points lie in the zero-free set $K^\ast$.  If
$P_j(re^{i\vartheta})=0$, then, because $p_j$ has real coefficients,
\[
 r^{N_j}
 =\left|\frac{p_j(re^{i\vartheta})}
 {p_j(r^{-1}e^{-i\vartheta})}\right|
 =\left|\frac{p_j(re^{i\vartheta})}
 {p_j(r^{-1}e^{i\vartheta})}\right|.
\]
Taking logarithms and using \eqref{eq:critical-radial-ratio} yields
$N_j|\log r|\le D'|\log r|$.  Thus $r=1$ once $N_j>D'$, which proves the
claim.
\end{proof}

\begin{theorem}[Critical-window containment]
\label{thm:critical-window-counting}
Let $k\to\infty$ through even weights and let $m=m(k)\in\Z_{\ge0}$ vary with
$k$.  Suppose that $\alpha_{k,m}\to1/2$.  Then, for all sufficiently large
$k$, uniformly over normalized level-one Hecke eigenforms $f\in S_k(\SLZ)$,
the polynomial $Q_{f,m}^{-}$ is nonzero, and its nonzero off-circle zero set
is either empty or consists of the four simple real zeros
\[
 \{\pm b_{f,m,k},\,\pm b_{f,m,k}^{-1}\},\qquad 0<b_{f,m,k}<1.
\]
\end{theorem}

\begin{proof}
Let $N=2M$.  By Proposition~\ref{prop:transition-moving-edge-quantitative},
$c_0\ne0$ for all sufficiently large $k$.  Hence $p_{f,m,k}$ and
$P_{f,m,k}$ are defined; write $p=p_{f,m,k}$ and $P=P_{f,m,k}$ throughout the
proof.  The sine estimate \eqref{eq:moving-p-quantitative} and
$\alpha_{k,m}\to1/2$ give this convergence uniformly in the normalized
eigenform.  More precisely, for every fixed $R>0$,
\[
 \sup_{\substack{f\in S_k(\SLZ)\ \mathrm{normalized}}}
 \sup_{|z|\le R}
 \left|p_{f,m,k}(z)-\sin(\pi z)\right|\longrightarrow0.
\]
In particular, $p(z)\to\sin(\pi z)$ locally uniformly, and every sufficiently
large-$k$ choice made below may be taken uniformly in $f$.

\smallskip
\noindent\emph{Step 1: disk zeros of $p$ and the four-zero budget.}
Choose a small fixed
$0<\delta<1/2$ and then choose numbers
$0<u_{\rm ann}<u_{\rm col}<\log(1+\delta)$ and $0<\vartheta_{\rm col}<\pi/4$.
We take $U_+$ and $U_-$ to be log-polar collars around $1$ and $-1$:
\[
 U_+=\{e^{u+i\vartheta}: |u|<u_{\rm col},\ |\vartheta|<\vartheta_{\rm col}\},
\]
\[
 U_-=\{-e^{u+i\vartheta}: |u|<u_{\rm col},\ |\vartheta|<\vartheta_{\rm col}\}.
\]
These collars are invariant under
$z\mapsto\bar z$ and $z\mapsto1/\bar z$ and are contained in $|z|<1+\delta$.
Choose also a small disk $U_0$ about $0$, disjoint from the collars, with
$\overline{U_0}\subset\{|z|<\varrho_0\}$ for some $\varrho_0<e^{-u_{\rm ann}}$.  The parameters are
chosen so that the closures of $U_0,U_+$ and $U_-$ are pairwise disjoint, their
boundaries are piecewise smooth, and no boundary contains a zero of
$\sin(\pi z)$.  Each of these three neighborhoods contains exactly one zero of
$\sin(\pi z)$, namely $0$, $1$ and $-1$, respectively.  Shrinking the parameters
if necessary, put
\[
 \Omega=\{z:|z|\le 1+\delta\}\setminus(U_0\cup U_+\cup U_-).
\]
The compact set $\Omega$ is zero-free for $\sin(\pi z)$, so
$\min_\Omega|\sin(\pi z)|=:c_\Omega>0$.  Local uniform convergence gives
$|p-\sin(\pi z)|<c_\Omega/2$ on $\Omega$ for all large $k$; hence
$p$ has no zero in $\Omega$.  Since the boundaries of $U_0,U_+$ and
$U_-$ are contained in $\Omega$, the same inequality gives the Rouch\'e margin on
each boundary.  Rouch\'e's theorem therefore compares $p$ with
$\sin(\pi z)$ and gives exactly one zero of $p$ in each neighborhood,
counted with multiplicity.  Therefore
$p$ has exactly these three zeros in $|z|<1+\delta$, and in particular at
most three zeros in $\overline\D$.  Cauchy's estimate on a fixed circle around
the origin also gives $p'(0)\to\pi$, so the zero at the origin is simple.

By Lemma~\ref{lem:odd-part-completion}, the completion $P$ is a nonzero real
scalar multiple of $Q_{f,m}^{-}$.  Since $p'(0)\ne0$,
Lemma~\ref{lem:odd-part-completion} also shows that $P$ is an odd polynomial of
degree $N-1$ with a simple zero at the origin.  Since $p$ has at most
three zeros in $\overline\D$ and, for all sufficiently large $k$, $3\le N/2$,
Lemma~\ref{lem:disk-zero-circle-count-boundary} gives at least $N-6$ zeros of
$P$ on $\T$, counted with multiplicity.  Together
with the simple zero at the origin, this leaves at most four nonzero zeros away
from $\T$, counted with multiplicity.

\smallskip
\noindent\emph{Step 2: localization of the possible off-circle zeros.}
We next show that every nonzero off-circle zero in the open unit disk must lie
in $U_+\cup U_-$; the exterior zeros are then their reciprocal partners.

First, consider the compact set
\[
 K_{\mathrm{int}}
 =\{z:|z|\le e^{-u_{\rm ann}}\}\setminus(U_0\cup U_+\cup U_-).
\]
It is zero-free for $\sin(\pi z)$ and stays a positive distance from the
origin.  Hence $p$ is bounded away from zero there for all large $k$, while
$1/z$ ranges over a fixed compact set and
$z^Np(1/z)=O(e^{-Nu_{\rm ann}})$.  Thus $P$ has no zeros in
$K_{\mathrm{int}}$.

Second, we handle the neighborhood $U_0$ of the origin.  Choose
$\varepsilon>0$ so small that $\{z:|z|\le\varepsilon\}\subset U_0$.
On the compact set
\[
 \mathcal K_{\mathrm{orig}}=\overline{U_0}\setminus\{z:|z|<\varepsilon\}
\]
the function $\sin(\pi z)$ is zero-free, so $|p(z)|\ge c_\varepsilon>0$
there for all large $k$.  Since $|z|\le\varrho_0<1$ on $\mathcal K_{\mathrm{orig}}$ and $1/z$ ranges over
a fixed compact subset of $\C$, local uniform convergence gives
$p(1/z)=O_\varepsilon(1)$ on $\mathcal K_{\mathrm{orig}}$; hence
$z^Np(1/z)=O_\varepsilon(\varrho_0^N)$.  Thus $P$ has no zeros in $\mathcal K_{\mathrm{orig}}$.
On the circle $|z|=\varepsilon$, local uniform convergence gives
$|p(z)|\gg\varepsilon$, while
$z^Np(1/z)=O_\varepsilon(\varepsilon^N)$.  Rouch\'e's theorem therefore
shows that $P$ has exactly one zero in $\{z:|z|<\varepsilon\}$.  Since $P$ is odd,
that zero is the origin, and it is simple.

Third, we exclude off-circle zeros in the annulus near $\T$ away from
$\pm1$.  Zeros in the closed subdisk $|z|\le e^{-u_{\rm ann}}$ outside
$U_0\cup U_+\cup U_-$ have already been excluded by the
$K_{\mathrm{int}}$ argument above, and the exterior zeros are their reciprocal
partners.  In the
annulus
\[
 \mathscr A_{\rm ann}=\{e^{-u_{\rm ann}}\le |z|\le e^{u_{\rm ann}}\},
\]
Lemma~\ref{lem:critical-annular-exclusion}, applied to
$p_j=p$, $N_j=N$, and $\sigma_j=(-1)^m$, shows that every zero of $P$
in $\mathscr A_{\rm ann}\setminus(U_+\cup U_-)$ lies on $\T$.  This application
is uniform in $f$: otherwise one could choose a sequence of normalized
eigenforms and violating zeros, and the displayed uniform compact convergence
would satisfy the hypotheses of Lemma~\ref{lem:critical-annular-exclusion}, a
contradiction.  Therefore every nonzero off-circle zero in the open disk lies
in $U_+\cup U_-$.  By signed reciprocity, the zeros outside the disk are the
reciprocals of these inside zeros.

\smallskip
\noindent\emph{Step 3: exclusion of nonreal off-circle orbits.}
The localization in Step 2 is used here to rule out possible orbit collapses on
the imaginary axis; without that localization the symmetry orbit below could be
smaller at special points.  The polynomial $P$ has real coefficients, is odd, and satisfies
\[
 P(z)=(-1)^m z^NP(1/z).
\]
Thus any nonzero zero $\zeta$ off $\T$ brings with it
all zeros in the orbit
\[
 \zeta,\quad -\zeta,\quad \bar\zeta,\quad -\bar\zeta,\quad
 \zeta^{-1},\quad -\zeta^{-1},\quad \bar\zeta^{-1},\quad -\bar\zeta^{-1}.
\]
For $U_+$ and $U_-$ chosen sufficiently small, every point in their union has
real part bounded away from $0$ in absolute value and is separated from the
imaginary axis.  We now exclude all possible collapses of this orbit for a
non-real off-circle zero.  The equalities $\zeta=\bar\zeta$ and
$\zeta=-\bar\zeta$ would force $\zeta$ to be real or purely imaginary,
respectively; the first is excluded by non-reality and the second by the
localization near $\pm1$.  The equalities $\zeta=\zeta^{-1}$ and
$\zeta=-\zeta^{-1}$ would force $\zeta^2=1$ or $\zeta^2=-1$, hence a real
point $\pm1$ or a purely imaginary point $\pm i$, again impossible for a
localized non-real off-circle zero.  Finally, $\zeta=\bar\zeta^{-1}$ is
equivalent to $|\zeta|=1$, and $\zeta=-\bar\zeta^{-1}$ would imply
$|\zeta|^2=-1$.  The same checks after multiplying by signs or conjugating show
that the displayed orbit contains eight distinct off-circle zeros.  This
contradicts the four-zero bound.  Hence all off-circle zeros are real.

\smallskip
\noindent\emph{Step 4: the real quartet and simplicity.}
If the off-circle set is nonempty, let $b\in(0,1)$ be the absolute value of a
real off-circle zero in $\D$.  Oddness and signed reciprocity force the four
distinct zeros
\[
 \{\pm b,\,\pm b^{-1}\}
\]
with equal multiplicity.  The four-zero multiplicity budget from Step~1 shows
that these are the only off-circle zeros and that each is simple.
\end{proof}

\section{The lower endpoint \texorpdfstring{$\alpha\to0$}{alpha to 0}}
\label{sec:lower-endpoint}

The lower separated phase in Theorem~\ref{thm:separated-moving-phase}
already treats sequences for which $\alpha_{k,m}\to0$ but
$k\alpha_{k,m}\to\infty$.  It remains to consider the complementary range
$k\alpha_{k,m}=O(1)$, which we state uniformly as
$\alpha_{k,m}\le\mathfrak a/k$ with $\mathfrak a$ fixed.  There the sine model
has size $O(k^{-1})$ on $\T$,
the same size as the absolute error in
Proposition~\ref{prop:transition-moving-edge-quantitative}, so a relative
estimate is required.  In this range the entire right-edge array, not only its
first $O(\log k)$ coefficients, carries powers of $\alpha_{k,m}$.  The weighted
estimate in Lemma~\ref{lem:transition-model-localization} supplies the required
relative control, beginning with the coefficient decay below.

\begin{lemma}[Lower-endpoint coefficient decay]
\label{lem:alpha-zero-weighted-tail}
Fix $\mathfrak a>0$.  There are constants $K_{\mathfrak a}\ge1$ and
$A_{\mathfrak a}\ge1$ such that the following holds.  Let
$k\ge K_{\mathfrak a}$ be even and write $k-2=2M$.  Let $m\ge0$ be an
integer, and let $f\in S_k(\SLZ)$ be a normalized Hecke eigenform.  If
\[
 \alpha:=\alpha_{k,m}\le \frac{\mathfrak a}{k},
\]
then $c_0\ne0$, and the right-edge coefficients in~\eqref{eq:cr-definition}
satisfy
\begin{equation}
 \left|\frac{c_r}{c_0}\right|
 \le \frac{(A_{\mathfrak a}\alpha)^r}{r!}
 \qquad(1\le r\le M).
 \label{eq:alpha-zero-coefficient-tail}
\end{equation}
Moreover,
\begin{equation}
 \frac{c_1}{c_0}=2\pi\alpha\left(1+O_{\mathfrak a}(k^{-1})\right).
 \label{eq:alpha-zero-first-coeff}
\end{equation}
The constants and the $O$-term are independent of $m$ and $f$.
\end{lemma}

\begin{proof}
Set $d=k-2=2M$ and $s_0=k-1$.  By
\eqref{eq:transition-denominator-lower} and $G(s_0)>0$, after increasing
$K_{\mathfrak a}$ if necessary, one has $c_0\ne0$ uniformly in $m$ and $f$.
Let $x_0$ be defined by~\eqref{eq:transition-saddle}, and put
$Y_0=e^{x_0}$.  Since $2\pi Y_0=s_0/\alpha$, the assumption gives
$Y_0\gg_{\mathfrak a}k^2$.  The curvature at the saddle satisfies
\[
\begin{aligned}
 B_0&=2\pi Y_0+\frac{m}{x_0^2} \\
    &\ge 2\pi Y_0=\frac{s_0}{\alpha}
     \ge \frac{k(k-1)}{\mathfrak a}
     \ge \frac{M^2}{\mathfrak a}.
\end{aligned}
\]
In particular, the curvature hypothesis in
Lemma~\ref{lem:transition-model-localization} holds with $H=M$.
Let $\varpi_0$ be the probability measure on $(0,\infty)$ with density
\[
 \mathcal J_m(s_0)^{-1}e^{s_0x-2\pi e^x}x^m\,dx.
\]
After increasing $K_{\mathfrak a}$ once more,
Lemma~\ref{lem:transition-model-localization} applies with $S=s_0$, $x_S=x_0$,
$B_S=B_0$, $H=M$, and
$d\varpi_S=d\varpi_0$.

\smallskip
\noindent\emph{The weighted Mellin bound.}\par
\noindent For $1\le r\le M$ and $n\ge1$, Lemma~\ref{lem:transition-model-localization}
gives
\begin{align}
 \frac{\mathcal J_{n,m}(s_0-r)}{\mathcal J_m(s_0)}
 &=\int e^{-rx}e^{-2\pi(n-1)e^x}\,d\varpi_0(x)
 \notag\\
 &\ll_{\mathfrak a} C_{\mathfrak a}^{\,r}e^{-rx_0}e^{-c_{\mathfrak a}(n-1)Y_0}
       +e^{-c_{\mathfrak a}B_0}e^{-2\pi(n-1)}.
 \label{eq:alpha-zero-Jn-bound}
\end{align}

\smallskip
\noindent\emph{Summation over Fourier modes.}\par
\noindent Multiplying by Deligne's bound~\cite{Deligne}
$|a_f(n)|\le \tau(n)n^{(k-1)/2}$, where $\tau$ is the divisor function, and
summing over $n$, the first term in~\eqref{eq:alpha-zero-Jn-bound} contributes
only its $n=1$ term up to an exponentially small error.  Indeed, after
decreasing $c_{\mathfrak a}>0$ if necessary, for $n\ge2$ and all sufficiently
large $k$ one has
\[
 \tau(n)n^{(k-1)/2}e^{-c_{\mathfrak a}(n-1)Y_0}
 \le \exp\{-c_{\mathfrak a}(n-1)Y_0/2\},
\]
using $\tau(n)\le n$, because $Y_0/k\to\infty$ uniformly and
$\sup_{n\ge2}\log n/(n-1)<\infty$.  Hence the whole $n\ge2$ part is
$O_{\mathfrak a}(e^{-c_{\mathfrak a}Y_0})$.

For the second term, put $\nu=(k+1)/2$.  Since $\tau(n)\le n$,
\[
 \sum_{n\ge1}\tau(n)n^{(k-1)/2}e^{-2\pi(n-1)}
 \le e^{2\pi}\sum_{n\ge1}n^\nu e^{-2\pi n}.
\]
For every $n\ge1$,
\[
 n^\nu e^{-\pi n}
 \le \sup_{x>0}x^\nu e^{-\pi x}
 =\left(\frac{\nu}{\pi e}\right)^\nu,
\]
because the derivative of $\nu\log x-\pi x$ vanishes at $x=\nu/\pi$ and its
second derivative is negative.  Consequently
\[
 \sum_{n\ge1}\tau(n)n^{(k-1)/2}e^{-2\pi(n-1)}
 \le e^{2\pi}\left(\frac{\nu}{\pi e}\right)^\nu
       \sum_{n\ge1}e^{-\pi n}
 \le e^{Ck\log k}
\]
for an absolute constant $C$.  Since $B_0\gg_{\mathfrak a}k^2$, the exponent
$Ck\log k-c_{\mathfrak a}B_0$ is at most $-c'_{\mathfrak a}B_0$ for all
sufficiently large $k$, after decreasing the positive constant once.  Thus
both errors are $O_{\mathfrak a}(e^{-c_{\mathfrak a}Y_0})$.

It remains to absorb this error into the $n=1$ scale uniformly for
$1\le r\le M$.  After increasing $C_{\mathfrak a}\ge1$ if necessary,
\[
\begin{aligned}
 \log\!\left(
  \frac{e^{-c_{\mathfrak a}Y_0}}
       {C_{\mathfrak a}^{\,r}e^{-rx_0}}
 \right)
 &=-c_{\mathfrak a}Y_0+rx_0-r\log C_{\mathfrak a} \\
 &\le -c_{\mathfrak a}Y_0+M\log Y_0
  =-c_{\mathfrak a}Y_0+O(k\log Y_0)
  \longrightarrow-\infty,
\end{aligned}
\]
because $Y_0\gg_{\mathfrak a}k^2$.  Therefore
\begin{equation}
 |\mathcal I_{f,m}(s_0-r)|
 \ll_{\mathfrak a} C_{\mathfrak a}^{\,r}e^{-rx_0}\mathcal J_m(s_0)
 \qquad(1\le r\le M).
 \label{eq:alpha-zero-main-I}
\end{equation}

\smallskip
\noindent\emph{Transfer to the coefficient array.}\par
\noindent Since $1+r\le s_0-r$ for $1\le r\le M$,
\[
\begin{aligned}
 |\mathcal I_{f,m}(1+r)|
 &\le \sum_{n\ge1}|a_f(n)|\mathcal J_{n,m}(1+r) \\
 &\le \sum_{n\ge1}|a_f(n)|\mathcal J_{n,m}(s_0-r) \\
 &\ll_{\mathfrak a}
 C_{\mathfrak a}^{\,r}e^{-rx_0}\mathcal J_m(s_0),
\end{aligned}
\]
where the last estimate is the absolute Fourier majorant established above.  The denominator lower bound
\eqref{eq:transition-denominator-lower} gives
\begin{equation}
 \left|\frac{\Lambda_f^{(m)}(s_0-r)}{\Lambda_f^{(m)}(s_0)}\right|
 \ll_{\mathfrak a} C_{\mathfrak a}^{\,r}e^{-rx_0}
 \qquad(1\le r\le M).
 \label{eq:alpha-zero-Lambda-tail}
\end{equation}
Combining~\eqref{eq:alpha-zero-Lambda-tail} with
\[
 \binom dr e^{-rx_0}
 \le \frac{(de^{-x_0})^r}{r!}
 \le \frac{(s_0e^{-x_0})^r}{r!}
 =\frac{(2\pi\alpha)^r}{r!}
\]
gives
\[
 \left|\frac{c_r}{c_0}\right|
 \ll_{\mathfrak a}\frac{(2\pi C_{\mathfrak a}\alpha)^r}{r!}.
\]
Since $r\ge1$, the implied constant can be absorbed into
$C_{\mathfrak a}^{\,r}$ after increasing $C_{\mathfrak a}$.  Taking
$A_{\mathfrak a}=2\pi C_{\mathfrak a}$ then gives
\eqref{eq:alpha-zero-coefficient-tail}.  Finally,
\eqref{eq:alpha-zero-first-coeff} is the case $r=1$ of
Proposition~\ref{prop:transition-near-edge-ratio}.
\end{proof}

\begin{theorem}[The lower endpoint]
\label{thm:alpha-zero-endpoint}
Fix $\mathfrak a>0$.  For all sufficiently large even $k$, uniformly over all
integers $m\ge0$ satisfying
\[
 \alpha_{k,m}\le \frac{\mathfrak a}{k}
\]
and all normalized level-one Hecke eigenforms $f$, the origin is a simple zero
of $Q_{f,m}^{-}$, and every nonzero zero of $Q_{f,m}^{-}$ is simple and lies on
$\T$.
\end{theorem}

\begin{proof}
Let $\alpha=\alpha_{k,m}$.  By
Lemma~\ref{lem:alpha-zero-weighted-tail}, $c_0\ne0$ for all sufficiently large
$k$, uniformly in $m$ and $f$.  Hence $p_{f,m,k}$ and $P_{f,m,k}$ are defined;
write $p=p_{f,m,k}$ and $P=P_{f,m,k}$ throughout the proof.
Fix $R>0$.  We spell out the coefficient estimate, since the relative bound
below divides by $\alpha$.  Write $\widetilde c_r=c_r$ for $r<M$ and
$\widetilde c_M=c_M/2$.  From the definition of $p$,
\[
 p(z)=\sum_{\substack{1\le r\le M\\ r\ {\rm odd}}}
 i^{r-1}\frac{\widetilde c_r}{c_0}z^r.
\]
For all sufficiently large $k$ we have $M>1$, so $\widetilde c_1=c_1$.
Lemma~\ref{lem:alpha-zero-weighted-tail} gives
\[
 \frac{i^{0}\widetilde c_1}{2\pi\alpha c_0}=1+O_{\mathfrak a}(k^{-1}),
\]
and, for $|z|\le R$,
\begin{align*}
 \sum_{\substack{3\le r\le M\\ r\ {\rm odd}}}
 \frac1{2\pi\alpha}\left|\frac{\widetilde c_r}{c_0}\right||z|^{r-1}
 &\ll_{R,\mathfrak a}\frac1\alpha
   \sum_{r\ge3}\frac{(A_{\mathfrak a}\alpha)^rR^{r-1}}{r!}  \\
 &\ll_{R,\mathfrak a}\alpha^2=O_{R,\mathfrak a}(k^{-2}).
\end{align*}
Therefore the analytic quotient $p(z)/(2\pi\alpha z)$, with its removable
value at $z=0$, satisfies
\begin{equation}
 \sup_{|z|\le R}\left|
 \frac{p(z)}{2\pi\alpha z}-1\right|=O_{R,\mathfrak a}(k^{-1}).
 \label{eq:alpha-zero-relative-p}
\end{equation}
Taking $R=2$, the supremum in~\eqref{eq:alpha-zero-relative-p} is less than
$1/2$ for all sufficiently large $k$.  Hence the analytic quotient is
zero-free on $|z|\le2$.  Thus $p$ has exactly one zero there, namely a simple
zero at the origin, and in particular has no zero on $\T$.

Put $N=2M$.  The estimate~\eqref{eq:alpha-zero-relative-p} gives
$p'(0)\ne0$.  Lemma~\ref{lem:odd-part-completion} therefore shows that $P$ is a
nonzero real scalar multiple of $Q_{f,m}^{-}$, has degree $N-1$, and has a
simple zero at the origin.  Lemma~\ref{lem:disk-zero-circle-count}, applied
with $r=1$, gives at least $N-2$ distinct zeros on $\T$.  Together with the
simple zero at the origin, these zeros exhaust the degree of $P$; hence all
nonzero zeros are simple and lie on $\T$.  The same conclusion holds for
$Q_{f,m}^{-}$.
\end{proof}

\section{The upper endpoint \texorpdfstring{$\alpha\to1$}{alpha to 1}}
\label{sec:upper-endpoint}

The upper separated phase in Theorem~\ref{thm:separated-moving-phase} already
handles approaches to $\alpha_{k,m}=1$ for which
$k(1-\alpha_{k,m})\to\infty$.  The exact endpoint $m=0$ is the classical
odd-period-polynomial case identified in
Lemma~\ref{lem:mzero-classical-period}.  We treat here the remaining positive
derivative orders $m=O(\log k)$.  In this range the real model zeros near
$1/2$ and $1$ move on the same or a smaller scale than the $O(k^{-1})$ error
in the zeroth-order approximation, so a first-order Gamma-derivative
expansion is required.  The Bell-polynomial estimates and the exponentially
small $L$-factor bound are proved in Appendix~\ref{sec:bell-appendix}; the
resulting first-order edge expansion is the input for the zero analysis below.

\subsection{The first-order edge model}

Within this subsection and Appendix~\ref{sec:bell-appendix}, we write $\mu$
for the derivative order, reserving $m$ for the global weight--derivative
notation.  Accordingly, $\mathcal D_{f,\mu}$, $q_{f,\mu,k}$, and
$p_{f,\mu,k}$ denote the objects of Section~\ref{sec:normalization-completion}
with $m$ replaced by $\mu$, and
\[
 c_{\mu,r}:=\mathcal D_{f,\mu}(k-1-r)\frac{(2\pi)^r}{r!}.
\]
Thus $c_{\mu,r}$ is the right-edge coefficient at derivative order $\mu$; in
particular, $c_{\mu,0}=\mathcal D_{f,\mu}(k-1)$.

Let $\psi=\Gamma'/\Gamma$ be the digamma function.  Define
\begin{equation}
 \mathfrak d_{k,\mu}:=\frac{\mu\psi'(k-1)}{\psi(k-1)-\log(2\pi)}.
 \label{eq:slow-edge-displacement}
\end{equation}
For every fixed $\mathfrak b>0$, uniformly for integers
$1\le\mu\le\mathfrak b\log k$, one has, for all sufficiently large $k$,
\[
 \mathfrak d_{k,\mu}>0,
 \qquad
 \mathfrak d_{k,\mu}\asymp_{\mathfrak b}\frac{\mu}{k\log k}.
\]
In particular, uniformly in this range,
\[
 \mathfrak d_{k,\mu}\longrightarrow0,
 \qquad
 \mathfrak d_{k,\mu}\gg_{\mathfrak b}\frac1{k\log k}.
\]
Throughout this section and Appendix~\ref{sec:bell-appendix}, the notation
$o_{\mathfrak b,R}(\mathfrak d_{k,\mu})$ means that the ratio to
$\mathfrak d_{k,\mu}$ tends to zero uniformly over all integers
$1\le\mu\le\mathfrak b\log k$ and all normalized level-one Hecke eigenforms,
with $\mathfrak b$ and $R$ fixed.  The notation $o_{\mathfrak b}$ has the
analogous meaning when the compact radius is fixed, and the same convention
applies with $\mu$ replaced by $m$.

\begin{proposition}[First-order edge expansion in the logarithmic range]
\label{prop:slow-first-order-edge}
Fix $\mathfrak b,R>0$.  There is $K=K(\mathfrak b,R)$ such that, for every even
$k\ge K$, every integer $\mu$ with $1\le\mu\le \mathfrak b\log k$, and every
normalized level-one Hecke eigenform $f$, one has $c_{\mu,0}\ne0$ and
\begin{equation}
 \sup_{|w|\le R}\left|
 \frac{q_{f,\mu,k}(w)}{c_{\mu,0}}
 -e^{2\pi w}(1-2\pi\mathfrak d_{k,\mu}w)
 \right|=o_{\mathfrak b,R}(\mathfrak d_{k,\mu}).
 \label{eq:slow-first-order-q}
\end{equation}
Consequently
\begin{equation}
 p_{f,\mu,k}(z)=\sin(2\pi z)-2\pi\mathfrak d_{k,\mu}z\cos(2\pi z)
 +o_{\mathfrak b,R}(\mathfrak d_{k,\mu}),
 \label{eq:slow-first-order-p}
\end{equation}
uniformly for $|z|\le R$.
\end{proposition}

Proposition~\ref{prop:slow-first-order-edge} is proved in
Appendix~\ref{sec:bell-appendix}.

\subsection{Zeros of the perturbed sine model}

\begin{lemma}[Zeros of the upper-endpoint model]
\label{lem:upper-endpoint-model-zeros}
For $\delta>0$ sufficiently small, set
\[
 \Phi_\delta(z)=\sin(2\pi z)-2\pi\delta z\cos(2\pi z).
\]
Then $\Phi_\delta$ has exactly five zeros in $|z|<5/4$, all simple and real, namely
\[
 0,\quad
 \pm\left(\frac12+\frac\delta2+O(\delta^2)\right),\qquad
 \pm\left(1+\delta+O(\delta^2)\right).
\]
Moreover $\Phi_\delta$ has no zero on $\T$, and
\[
 \min_{|z|=1}|\Phi_\delta(z)|\gg\delta,
 \qquad
 \min_{|z|=5/4}|\Phi_\delta(z)|\gg1.
\]
\end{lemma}

\begin{proof}
The zeros of $\sin(2\pi z)$ in $|z|<5/4$ are
$-1,-1/2,0,1/2,1$, all simple and separated by a fixed distance.  Choose pairwise
disjoint conjugation-invariant disks $D_a$ around
$a\in\{-1,-1/2,0,1/2,1\}$, all contained in $|z|<5/4$.  On each boundary
$\partial D_a$ the function $\sin(2\pi z)$ is bounded away from zero, while the
perturbation $-2\pi\delta z\cos(2\pi z)$ is $O(\delta)$.  Hence, for
$\delta$ small, Rouch\'e's theorem gives exactly one zero of $\Phi_\delta$ in each
$D_a$, counted with multiplicity.  On the compact complement of these disks in
$\{|z|\le5/4\}$, the same lower bound for $|\sin(2\pi z)|$ shows that there
are no further zeros of $\Phi_\delta$ in $|z|<5/4$.

Taylor expansion at a nonzero point $a\in\{\pm1/2,\pm1\}$ gives
\[
 0=2\pi\cos(2\pi a)(z-a)-2\pi\delta a\cos(2\pi a)
   +O((z-a)^2+\delta|z-a|).
\]
The unique zero in $D_a$ tends to $a$ as $\delta\to0$, so this equation first
gives $z-a=O(\delta)$; substituting this estimate back gives
$z-a=\delta a+O(\delta^2)$.  Since $D_a$ is conjugation-invariant and the zero
in it is unique, that zero is real; moreover
$\Phi_\delta'(a+\delta a+O(\delta^2))=2\pi\cos(2\pi a)(1+O(\delta))\ne0$, so it is
simple.  The zero at the origin remains exactly at the origin by oddness and is
simple because $\Phi_\delta'(0)=2\pi(1-\delta)$.  This proves the displayed
locations and the asserted simplicity and reality.  Since the circle $|z|=5/4$
stays a fixed positive distance from the zero set $\frac12\Z$ of
$\sin(2\pi z)$, the same compactness argument gives
$\min_{|z|=5/4}|\Phi_\delta(z)|\gg1$ for small $\delta$.

It remains only to record the boundary separation on $\T$.  Away from small arcs around
$\pm1$ on $\T$, the preceding compactness argument gives a positive lower bound.
Near $1$, write $z=e^{i\vartheta}$.  The expansion
\[
 \Phi_\delta(e^{i\vartheta})
 =2\pi(i\vartheta-\delta)
  +O(\vartheta^2+\delta|\vartheta|+\delta^2)
\]
gives $|\Phi_\delta(e^{i\vartheta})|\gg |\vartheta|+\delta$ for
$|\vartheta|$ and $\delta$ small.  The same argument at $-1$ gives the identical
bound.  Hence the minimum on $\T$ is $\gg\delta$.
\end{proof}

We now return to the global derivative-order notation $m$; thus
$c_r=c_{m,r}$ and $\mathfrak d_{k,m}$ is given by
\eqref{eq:slow-edge-displacement} with $\mu=m$.

\newpage
\begin{theorem}[The upper endpoint]
\label{thm:alpha-one-endpoint}
Fix $\mathfrak b>0$.  There is $K_{\mathfrak b}\ge1$ such that, for every even
$k\ge K_{\mathfrak b}$, every integer $1\le m\le \mathfrak b\log k$, and every
normalized level-one Hecke eigenform $f$, the odd part $Q_{f,m}^{-}$ has a
simple zero at $0$ and four further simple zeros
\[
 \pm b_{f,m,k},\qquad \pm b_{f,m,k}^{-1},\qquad 0<b_{f,m,k}<1,
\]
forming a single real reciprocal quartet; all remaining zeros are simple and
lie on $\T$.  Moreover, uniformly in the same range as $k\to\infty$,
\[
 b_{f,m,k}=\frac12+\frac{\mathfrak d_{k,m}}2+o_{\mathfrak b}(\mathfrak d_{k,m}),
 \qquad
 b_{f,m,k}^{-1}=2-2\mathfrak d_{k,m}+o_{\mathfrak b}(\mathfrak d_{k,m}).
\]
\end{theorem}

\begin{proof}
\smallskip
\noindent\emph{Zeros of the edge polynomial.}
By Proposition~\ref{prop:slow-first-order-edge}, the normalizing coefficient
$c_0=\mathcal D_{f,m}(k-1)$ is nonzero for all sufficiently large $k$.  Hence
$p_{f,m,k}$ and $P_{f,m,k}$ are defined.  Write
$p=p_{f,m,k}$ and $P=P_{f,m,k}$, put $\delta=\mathfrak d_{k,m}$, and set
$\Phi_\delta(z)=\sin(2\pi z)-2\pi\delta z\cos(2\pi z)$.  The same proposition gives,
uniformly for $|z|\le5/4$,
\[
 p(z)=\Phi_\delta(z)+o_{\mathfrak b}(\delta).
\]
By Cauchy's estimate on a fixed small circle around the origin, the same
compact-uniform estimate gives
\[
 p'(0)=\Phi_\delta'(0)+o_{\mathfrak b}(\delta)
       =2\pi+O_{\mathfrak b}(\delta).
\]
Lemma~\ref{lem:upper-endpoint-model-zeros}, followed by Rouch\'e's theorem on
$|z|=5/4$, shows first that $p$ and $\Phi_\delta$ have the same total
number, namely five, of zeros in $|z|<5/4$.

Fix $\varepsilon_*>0$.  Around each of the five model zeros in this disk, take
a disk of radius $\varepsilon_*\delta$; for all sufficiently large $k$ these
disks are pairwise disjoint.  Simplicity of the model zeros and their
expansions in Lemma~\ref{lem:upper-endpoint-model-zeros} give
\[
 |\Phi_\delta(z)|\gg_{\varepsilon_*}\delta
\]
on the boundary of every such disk.  Since
$p-\Phi_\delta=o_{\mathfrak b}(\delta)$ uniformly there, Rouch\'e's theorem
gives exactly one zero of $p$ in each disk.  Since $p$ is odd, the zero in the
central disk is exactly $0$.  Thus, uniformly in $m$ and $f$, the five zeros
are
\[
 0,
 \qquad
 \pm\left(\frac12+\frac{\delta}{2}+O(\delta^2)+O(\varepsilon_*\delta)\right),
 \qquad
 \pm\left(1+\delta+O(\delta^2)+O(\varepsilon_*\delta)\right).
\]
Because each zero lies within distance $\varepsilon_*\delta$ of the
corresponding model zero, its normalized deviation from the corresponding
first-order location has limsup at most $\varepsilon_*$, uniformly in $m$ and
$f$.  Letting $\varepsilon_*\downarrow0$ gives
\[
 0,
 \qquad
 \pm\left(\frac12+\frac{\delta}{2}+o_{\mathfrak b}(\delta)\right),
 \qquad
 \pm\left(1+\delta+o_{\mathfrak b}(\delta)\right).
\]
Each localization disk is invariant under conjugation and contains exactly one
zero of the real polynomial $p$; hence all five localized zeros are real.  The
lower bound on $\T$ in Lemma~\ref{lem:upper-endpoint-model-zeros}, together with
$p-\Phi_\delta=o_{\mathfrak b}(\delta)$, shows that $p$ has no zero on $\T$.
For all sufficiently large $k$, the pair near
$\pm(1+\delta)$ lies outside $\overline{\D}$, whereas the origin and the pair
near $\pm(1/2+\delta/2)$ lie in $\D$.  Hence $p$ has exactly three zeros in
$\D$ and no zero on $\T$.

\smallskip
\noindent\emph{Passage to the signed-reciprocal completion.}
Let $N=2M$.  For all sufficiently large $k$, one has $M\ge2$; since
$p'(0)\ne0$, Lemma~\ref{lem:odd-part-completion} shows that $P\not\equiv0$,
that $P$ has degree $N-1$, and that the origin is a simple zero of $P$.  Also
$3\le N/2$, so Lemma~\ref{lem:disk-zero-circle-count}, applied with $r=3$,
gives at least $N-6$ distinct unit-circle zeros of $P$.  Consider the disk
$D_+$ centered at
$1/2+\delta/2$ with radius $\varepsilon_{\mathrm{comp}}\delta$, where
$\varepsilon_{\mathrm{comp}}>0$ is fixed and small; define $D_-=-D_+$.  The corresponding model
zeros differ from the disk centers by
$O(\delta^2)=o(\varepsilon_{\mathrm{comp}}\delta)$.  Hence, on
$\partial D_+\cup\partial D_-$, the model-zero expansion gives
$|\Phi_\delta(z)|\gg_{\varepsilon_{\mathrm{comp}}}\delta$, and the first-order edge estimate
therefore gives $|p(z)|\gg_{\varepsilon_{\mathrm{comp}}}\delta$.  By
Proposition~\ref{prop:slow-first-order-edge} with $R=3$, we have
$p(w)=O_{\mathfrak b}(1)$ for $|w|\le3$.  For
$z\in D_+\cup D_-$ and large $k$, $|1/z|\le3$ and $|z|\le3/5$, hence
\[
 z^Np(1/z)=O_{\mathfrak b}((3/5)^N)=o_{\mathfrak b}(\delta),
\]
because $\delta\gg_{\mathfrak b}(k\log k)^{-1}$.  Rouch\'e's theorem therefore
gives one zero of $P$ in $D_+$ and one in $D_-$.  The disks are invariant under
conjugation and $P$ has real coefficients, so these unique zeros are real.
For every fixed $\varepsilon_{\mathrm{comp}}>0$, their normalized deviations from
$\pm(1/2+\delta/2)$ have limsup at most $\varepsilon_{\mathrm{comp}}$, uniformly in $m$ and
$f$; letting $\varepsilon_{\mathrm{comp}}\downarrow0$ gives the two inner zeros
\[
 \pm\left(\frac12+\frac\delta2+o_{\mathfrak b}(\delta)\right).
\]
Signed reciprocity supplies their reciprocal partners, whose positive member is
\[
 \left(\frac12+\frac\delta2+o_{\mathfrak b}(\delta)\right)^{-1}
 =2-2\delta+o_{\mathfrak b}(\delta).
\]
The simple zero at the origin, the two localized inner zeros,
and their two reciprocal partners give five distinct zeros.  Together with the
at least $N-6$ distinct unit-circle zeros, they give $N-1$ distinct zeros and
therefore exhaust the degree of $P$.  Hence these are all the zeros of $P$ and
all are simple.  Lemma~\ref{lem:odd-part-completion} transfers the conclusion
to $Q_{f,m}^{-}$.
\end{proof}

\section{The derivative-order transition and the large-weight theorem}
\label{sec:derivative-axis-consequences}

We convert the phase parameter $\alpha_{k,m}$ to the critical derivative order
$m_c(k)=(k-1)\log((k-1)/\pi)$, quantify the logarithmic collision window,
and then assemble the separated, critical, and endpoint regimes into the
uniform large-weight theorem.

\subsection{Macroscopic transition and the logarithmic collision window}

\begin{proof}[Proof of Theorem~\ref{thm:intro-logarithmic-transition}]
Write $s=k-1$ and let $x_0=x_0(k,m)$.  The saddle equation gives
$x_0\ge\log(s/(2\pi))\asymp\log s$.  Since $m=O(s\log s)$, it follows that
$m/x_0=O(s)$, and the same equation then gives $e^{x_0}=O(s)$.  Thus
\[
 \log\frac{s}{2\pi}\le x_0\le \log(Cs)
\]
for some constant $C=C(\theta)$ and all large $k$, and hence
$x_0/\log(s/\pi)\to1$.  Therefore
\[
 \frac{m}{sx_0}
 =\frac{m}{s\log(s/\pi)}\frac{\log(s/\pi)}{x_0}
 \longrightarrow\theta,
\]
and consequently
\[
 \alpha_{k,m}=\frac{s}{s+m/x_0}\longrightarrow\frac1{1+\theta}.
\]
For $\theta<1$ this limit is strictly larger than $1/2$ and strictly smaller
than $1$, so the upper separated phase in
Theorem~\ref{thm:separated-moving-phase} applies and gives the displayed limits.
For $\theta>1$ the limit is strictly smaller than $1/2$ and positive, so the
lower separated phase applies.  For $\theta=1$ the parameter tends to $1/2$, and
Theorem~\ref{thm:critical-window-counting} gives nonvanishing and the stated
critical classification.  Since $\alpha_{k,m}$ depends only on $k$ and $m$, and
all cited phase theorems are uniform over normalized eigenforms, the conclusions
are uniform in $f$.
\end{proof}

We next prove the collar lemma deferred from Section~\ref{sec:moving-phases}.
It converts distance from $m_c(k)$ on the derivative-order axis into the boundary
margin $|\alpha_{k,m}-1/2|$.

\begin{proof}[Proof of Lemma~\ref{lem:critical-collar-width}]
Set
\[
 s=k-1,\qquad x_c=\log\frac{s}{\pi},\qquad
 m_c(k)=sx_c,\qquad \Delta_k=m-sx_c.
\]
Let $x_0$ be the saddle associated with $m$, and write
\[
 \mathscr M_s(x)=x(2\pi e^x-s),\qquad \alpha_s(x)=\frac{s}{2\pi e^x}.
\]
Then $\mathscr M_s(x_c)=m_c(k)$ and $\alpha_s(x_c)=1/2$.  For the two
one-sided assertions it suffices to argue along arbitrary subsequences.  Under
either one-sided hypothesis, after passing to a further subsequence, either
$|\Delta_k|\ge \xi s x_c$ for some fixed $\xi>0$, or
$|\Delta_k|=o(sx_c)$.

\smallskip
\noindent\emph{Macroscopic separation.}
In the first case the parameter is already separated from the critical value.
Indeed, if
$\Delta_k\ge \xi s x_c$, then monotonicity in Lemma~\ref{lem:alpha-monotone-m}
gives
\[
 \alpha_{k,m}\le \alpha_{k,(1+\xi)sx_c}.
\]
Put $m_+=(1+\xi)sx_c$, and let $x_+$ be its saddle.  Then
$x_+/x_c\to1$.  Indeed, the saddle equation gives
$x_+\ge x_c-\log2$, whereas $x_+\ge(1+\varepsilon)x_c$ for some fixed
$\varepsilon>0$ would make $2\pi e^{x_+}/s$ grow like $s^\varepsilon$,
contradicting $m_+\asymp sx_c$.  Consequently,
\[
 \alpha_{k,m_+}
 =\frac{1}{1+m_+/(sx_+)}
 =\frac{1}{1+(1+\xi)x_c/x_+}
 \longrightarrow \frac{1}{2+\xi}<\frac12.
\]
Similarly, on the pre-critical side, after decreasing $\xi$ if necessary so
that $0<\xi<1$, put $m_-=(1-\xi)sx_c$ and let $x_-$ be its saddle.  The same
argument gives $x_-/x_c\to1$, and $\Delta_k\le-\xi s x_c$ yields
\[
\begin{aligned}
 \alpha_{k,m}&\ge \alpha_{k,m_-},\\
 \alpha_{k,m_-}
 &=\frac{1}{1+m_-/(sx_-)}
  =\frac{1}{1+(1-\xi)x_c/x_-},\\
 &\longrightarrow \frac{1}{2-\xi}>\frac12.
\end{aligned}
\]
Thus $\alpha_{k,m}$ is separated from $1/2$ on the corresponding side, and the
asserted $k$-growth is immediate.

\smallskip
\noindent\emph{Near-critical linearization.}
In the second case, $|\Delta_k|=o(sx_c)$.  We first note that
the displacement $t:=x_0-x_c$ satisfies $t=o(1)$.  Indeed,
\[
 \mathscr M_s''(x)=2\pi e^x(2+x)>0\qquad(x>0),
\]
so $\mathscr M_s'$ is increasing.  For any fixed $\varepsilon>0$,
\[
 t\ge\varepsilon
 \quad\Longrightarrow\quad
 \Delta_k=\int_{x_c}^{x_c+t}\mathscr M_s'(u)\,du
 \ge t\mathscr M_s'(x_c)
 \ge \varepsilon \mathscr M_s'(x_c)
 =\varepsilon s(1+2x_c),
\]
whereas, for $0<\varepsilon<1$,
\[
 t\le-\varepsilon
 \quad\Longrightarrow\quad
 -\Delta_k\ge \int_{x_c-\varepsilon}^{x_c}\mathscr M_s'(u)\,du
 \ge \varepsilon \mathscr M_s'(x_c-\varepsilon)\gg_\varepsilon sx_c.
\]
Both alternatives contradict $|\Delta_k|=o(sx_c)$.

Taylor expansion at $x_c$ now gives
\[
 \mathscr M_s'(x_c)=s(1+2x_c),\qquad \mathscr M_s''(x_c)=2s(2+x_c),
\]
and hence, if $x_0=x_c+t$,
\[
 \Delta_k=s(1+2x_c)t+O(sx_c t^2).
\]
Since $t=o(1)$ and $s(1+2x_c)\asymp sx_c$, this relation first gives
$t=O(|\Delta_k|/(sx_c))$; substituting this bound into the quadratic remainder
yields
\[
 t=\frac{\Delta_k}{s(1+2x_c)}
   +O\left(\frac{\Delta_k^2}{s^2x_c^2}\right).
\]
Also
\[
 \alpha_s(x_c+t)=\frac12 e^{-t}=\frac12-\frac t2+O(t^2).
\]
Combining the two estimates gives \eqref{eq:alpha-critical-linearization}.
Under either one-sided near-critical hypothesis, the error term in that formula
is lower order than its linear term.  Since $s\sim k$ and
$x_c\sim\log k$, we therefore have
\[
 k\left|\alpha_{k,m}-\frac12\right|
 \asymp \frac{|\Delta_k|}{\log k}\longrightarrow\infty.
\]
The sign of $\alpha_{k,m}-1/2$ is opposite to the sign of $\Delta_k$, which gives
the two asserted one-sided conclusions.
\end{proof}

\begin{corollary}[Super-logarithmic separation from the critical order]
\label{cor:logarithmic-window}
Let $k\to\infty$ through even weights, let $m=m(k)\in\Z_{\ge1}$, and set
\[
 m_c(k)=(k-1)\log\frac{k-1}{\pi}.
\]
The following conclusions hold uniformly over normalized level-one Hecke
eigenforms.

If
\begin{equation}
 \frac{m_c(k)-m(k)}{\log k}\to+\infty,
 \label{eq:precritical-log-window}
\end{equation}
then, for all sufficiently large $k$, $Q_{f,m}^{-}$ has a simple zero at $0$,
and its nonzero off-circle zeros form exactly one real reciprocal quartet whose four
zeros are simple.  All remaining nonzero zeros are simple and lie on $\T$.  In
the near-critical subrange
$m_c(k)-m(k)=o(k\log k)$, the inner positive zero satisfies
\begin{equation}
 1-b_{f,m,k}\sim
 \frac{m_c(k)-m(k)}{(k-1)\bigl(1+2\log((k-1)/\pi)\bigr)}.
 \label{eq:log-window-inner-zero}
\end{equation}

If
\begin{equation}
 \frac{m(k)-m_c(k)}{\log k}\to+\infty,
 \label{eq:postcritical-log-window}
\end{equation}
then, for all sufficiently large $k$, the origin is simple and every nonzero
zero of $Q_{f,m}^{-}$ is simple and lies on $\T$.
\end{corollary}

\begin{proof}
Write $\alpha_k:=\alpha_{k,m(k)}$.

\smallskip
\noindent\emph{The pre-critical phase.}
Assume first \eqref{eq:precritical-log-window}.  Lemma~\ref{lem:critical-collar-width} gives
\[
 \alpha_k>\frac12,
 \qquad
 k\bigl(\alpha_k-\tfrac12\bigr)\to+\infty.
\]
We rule out an infinite subsequence on which the asserted pre-critical zero
pattern fails.  On any such subsequence, either a further subsequence has
$m=O(\log k)$, in which case the upper endpoint theorem
(Theorem~\ref{thm:alpha-one-endpoint}) gives the stated quartet conclusion, or a
further subsequence has $m/\log k\to+\infty$.  In the latter case,
$k(1-\alpha_k)\to+\infty$.  Indeed, the pre-critical inequality gives
$m<m_c(k)\ll k\log k$, and always $x_0\ge\log((k-1)/(2\pi))$.  Hence
$m/x_0\ll k$, so the saddle equation gives $e^{x_0}\ll k$ and therefore
$x_0=O(\log k)$.  With $u=m/x_0$ and $m/\log k\to+\infty$, we get
$u\to+\infty$, and
\[
 1-\alpha_k=\frac{u}{k-1+u}.
\]
Since $u\ll k$ and $u\to+\infty$, this identity gives
$k(1-\alpha_k)\asymp u\to+\infty$.  Together with
$k(\alpha_k-1/2)\to+\infty$, this is the upper separated phase condition.
Theorem~\ref{thm:separated-moving-phase} again gives exactly one real reciprocal
quartet, all four of whose zeros are simple.  Thus the pre-critical assertion
holds for all sufficiently large $k$.

\smallskip
\noindent\emph{The near-critical location formula.}
If, in addition, $m_c(k)-m(k)=o(k\log k)$, then
Lemma~\ref{lem:critical-collar-width} gives
\[
 \alpha_k-\frac12\sim
 \frac{m_c(k)-m(k)}{2(k-1)\bigl(1+2\log((k-1)/\pi)\bigr)}.
\]
In this subrange, $m=m_c(k)+o(k\log k)\sim k\log k$, so
$m/\log k\to+\infty$; hence the upper separated-phase hypotheses have already
been verified above.  The separated location formula gives
\[
 b_{f,m,k}=\frac1{2\alpha_k}+O(k^{-1}).
\]
The two assumptions imply $\alpha_k-1/2\to0$, while
\eqref{eq:precritical-log-window} gives
$k(\alpha_k-1/2)\to+\infty$.  Hence the $O(k^{-1})$ term is
$o(\alpha_k-1/2)$, and
\[
 1-b_{f,m,k}\sim 2(\alpha_k-1/2).
\]
This proves \eqref{eq:log-window-inner-zero}.

\smallskip
\noindent\emph{The post-critical phase.}
Assume now \eqref{eq:postcritical-log-window}.  Lemma~\ref{lem:critical-collar-width}
gives
\[
 \alpha_k<\frac12,
 \qquad
 k\bigl(\tfrac12-\alpha_k\bigr)\to+\infty.
\]
Again suppose that the asserted post-critical conclusion fails infinitely often.
On a failing subsequence, either a further subsequence satisfies
$\alpha_k\le \mathfrak a/k$ for some fixed $\mathfrak a>0$, in which case the
lower endpoint theorem, Theorem~\ref{thm:alpha-zero-endpoint}, gives the
unit-circle conclusion, or a further subsequence has
$k\alpha_k\to+\infty$.  In the latter case the lower separated phase in
Theorem~\ref{thm:separated-moving-phase} applies.  Both alternatives contradict
failure, so the post-critical assertion holds for all sufficiently large $k$.
\end{proof}

\begin{remark}[Collision-window width]
Corollary~\ref{cor:logarithmic-window} does not determine the internal scaling
law when $m-m_c(k)=O(\log k)$.  Equivalently, after passing to subsequences, all
regimes not already resolved by the corollary are confined to
\[
 |m-m_c(k)|=O(\log k).
\]
Along every super-logarithmically separated pre-critical sequence the real
quartet is present, while along every such post-critical sequence it is absent.
\end{remark}

\begin{remark}[The upper-endpoint regime $m/m_c(k)\to0$]
Suppose that $m/m_c(k)\to0$.  After passing to subsequences, the positive
derivative orders are covered by two mechanisms.  If $m/\log k\to\infty$, then
the upper separated phase applies and the exceptional quartet satisfies
\[
 b_{f,m,k}\to\frac12,
 \qquad b_{f,m,k}^{-1}\to2.
\]
If $m=O(\log k)$, then along the subsequence there is a fixed $C>0$ such that
$1\le m\le C\log k$, and the upper endpoint theorem gives the sharper
first-order displacement
\[
 b_{f,m,k}=\frac12+\frac12\mathfrak d_{k,m}+o_C(\mathfrak d_{k,m}),
 \qquad
 \mathfrak d_{k,m}=\frac{m\psi'(k-1)}{\psi(k-1)-\log(2\pi)},
\]
where $\psi=\Gamma'/\Gamma$ is the digamma function.  The ordinary case $m=0$
is the classical period-polynomial regime; it is covered by \CFI\ and is
separate from this derivative-side first-order displacement.
\end{remark}

\subsection{A fixed-order consequence}

\begin{corollary}[Fixed-order quartet asymptotics]
Fix an integer $m\ge1$.  For all sufficiently large even weights $k$, uniformly
over normalized level-one Hecke eigenforms $f\in S_k(\SLZ)$, $Q_{f,m}^{-}$ has
the simple zero $0$, and its nonzero off-circle zeros form one real reciprocal quartet
whose four zeros are simple.  All remaining zeros are simple and lie on $\T$.
Moreover the inner positive exceptional zero satisfies
\[
 b_{f,m,k}
 =\frac12+
   \frac{m\psi'(k-1)}{2\bigl(\psi(k-1)-\log(2\pi)\bigr)}
   +o\!\left(\frac{m}{k\log k}\right).
\]
\end{corollary}

\begin{proof}
This is the case $m=O(1)$ of Theorem~\ref{thm:alpha-one-endpoint}.  The
displacement estimate recorded before Proposition~\ref{prop:slow-first-order-edge}
gives
\[
 \mathfrak d_{k,m}\asymp\frac{m}{k\log k}
\]
for fixed $m$, so the uniform $o(\mathfrak d_{k,m})$ term in that theorem is
$o(m/(k\log k))$.
\end{proof}

\subsection{Assembly of the phase regimes}

The preceding results cover complementary ranges of the phase parameter.  A
subsequence exhaustion now assembles them into the theorem uniform in all
integer derivative orders.

\begin{proof}[Proof of Theorem~\ref{thm:intro-main}]
We prove the uniform formulation by contradiction.  If no such $K_0$ existed,
there would be a sequence of even weights $k\to\infty$, derivative orders
$m=m(k)\in\Z_{\ge0}$, and normalized Hecke eigenforms $f\in S_k(\SLZ)$
for which the asserted conclusion fails.

\smallskip
\noindent\emph{The classical and logarithmic-order ranges.}
The case $m=0$ is covered by Lemma~\ref{lem:mzero-classical-period} and the
theorem of \CFI~\cite{ConreyFarmerImamoglu}.  After discarding a subsequence we
may therefore assume $m\ge1$.  Passing to a subsequence, either $m=O(\log k)$ or
$m/\log k\to\infty$.  In the first case
Theorem~\ref{thm:alpha-one-endpoint} contradicts failure of the conclusion.  We
may therefore work on a counterexample subsequence for which
$m/\log k\to\infty$.

\smallskip
\noindent\emph{Reduction to the phase parameter.}
The saddle lower bound gives
\[
 x_0\ge\log\frac{k-1}{2\pi}>1
\]
for all sufficiently large $k$.  Hence $m/x_0\le m$, and the saddle equation
gives
\[
 2\pi e^{x_0}=k-1+\frac{m}{x_0}\le k+m,
 \qquad
 x_0\ll\log(k+m).
\]
Moreover, $m/\log k\to\infty$ implies $m/\log(k+m)\to\infty$: if $m\le k$,
then $\log(k+m)\le\log(2k)$, whereas if $m>k$, then
$\log(k+m)\le\log(2m)$ and $m/\log m\to\infty$.  Therefore, with
$u=m/x_0$, we have $u\to\infty$.  Since
\[
 \alpha_{k,m}=\frac{k-1}{k-1+u},
 \qquad
 1-\alpha_{k,m}=\frac{u}{k-1+u},
\]
it follows that
\[
 k(1-\alpha_{k,m})
 =\frac{ku}{k-1+u}
 \ge\frac12\min\{k,u\}
 \longrightarrow\infty.
\]

\smallskip
\noindent\emph{Classification of the limiting phase.}
Passing to a further subsequence, either
$\alpha_{k,m}\le \mathfrak a/k$ for some fixed $\mathfrak a>0$, in which case
Theorem~\ref{thm:alpha-zero-endpoint} applies, or $k\alpha_{k,m}\to\infty$.  In
the latter case, pass again to a subsequence on which $\alpha_{k,m}$ has a limit
in $[0,1]$.  If the limit is $0$, then the lower separated phase applies because
$k\alpha_{k,m}\to\infty$.  If the limit is $1/2$, then
Theorem~\ref{thm:critical-window-counting} gives the asserted conclusion.  If
the limit is different from $0$ and $1/2$, then the parameter is separated from
$1/2$, and the lower or upper separated phase applies; in the upper case the
preceding estimate supplies the required endpoint margin
$k(1-\alpha_{k,m})\to\infty$.  Every possible counterexample subsequence
therefore contains a further subsequence on which the full conclusion holds, a
contradiction.
\end{proof}

\section{Fixed weight, large derivative order, and finite exceptions}
\label{sec:fixed-weight-large-derivative}

The large-weight direction is now complete.  We fix an even weight $k$ and a
normalized level-one Hecke eigenform $f$, and let $m\to\infty$.  The split-Mellin
moment with parameter $k-2$ then dominates every moment with a smaller Mellin
parameter.  Consequently, according to the parity of $m$, the normalized
reduced odd part is asymptotic to
\[
 z^{k-4}+(-1)^m.
\]
The resulting root localization, together with real coefficients and signed
reciprocity, then forces all nonzero zeros onto $\T$.  This argument is
logically independent of the moving-saddle analysis and closes the
remaining unbounded direction needed for the finite-exception corollary.

\subsection{Dominance of the outer logarithmic moments}

We use the split-Mellin and model moments introduced in
\eqref{eq:log-moment}--\eqref{eq:model-moments}.

\begin{lemma}[Logarithmic-moment edge dominance]
\label{lem:log-moment-dominance}
Fix an even weight $k$ and a normalized level-one Hecke eigenform
$f\in S_k(\SLZ)$.  Let $S_+,S_-\in\R$ with $S_+>S_-$.  Then
$\mathcal I_{f,m}(S_+)>0$ for all sufficiently large $m$, and
\begin{equation}
 \frac{\mathcal I_{f,m}(S_-)}{\mathcal I_{f,m}(S_+)}\longrightarrow0
 \qquad(m\to\infty).
 \label{eq:moment-ratio-zero}
\end{equation}
\end{lemma}

\begin{proof}
Since $f$ is normalized, $a_f(1)=1$, and the Fourier coefficients of $f$ are
real.  Absolute convergence of the Fourier expansion for $y\ge1$ gives
\[
 f(iy)=e^{-2\pi y}+E_f(y),\qquad
 E_f(y)=O_f(e^{-4\pi y}).
\]
Thus $f(iy)$ is real for $y>0$.  After increasing $Y$, we may assume
$|E_f(y)|\le \frac12e^{-2\pi y}$ for $y\ge Y$, so in particular
$f(iy)>0$ on $[Y,\infty)$.

\smallskip
\noindent\emph{Escape from compact sets.}
Fix a real parameter $S_\ast$.  For every fixed $R>1$, the contribution of
$[1,R]$ to $\mathcal J_m(S_\ast)$ is
$O_{S_\ast,R}((\log R)^m)$.  Choose $Y_1>\max\{R,Y\}$ and use the interval
$[Y_1,Y_1+1]$ to obtain
\[
 \mathcal J_m(S_\ast)
 \ge
 \int_{Y_1}^{Y_1+1}
 e^{-2\pi y}y^{S_\ast}(\log y)^m\,\frac{dy}{y}
 \gg_{Y_1,S_\ast}(\log Y_1)^m.
\]
Since $\log Y_1>\log R$, every fixed compact interval contributes
$o(\mathcal J_m(S_\ast))$ as $m\to\infty$.

\smallskip
\noindent\emph{Dominance of the first Fourier mode.}
Given $\varepsilon>0$, choose $R\ge Y$ so large that
$|E_f(y)|\le\varepsilon e^{-2\pi y}$ for $y\ge R$.  On the compact interval
$[1,R]$,
\[
 \int_1^R |E_f(y)|y^{S_\ast}(\log y)^m\,\frac{dy}{y}
 =O_{f,S_\ast,R}((\log R)^m)
 =o(\mathcal J_m(S_\ast)),
\]
while the tail is bounded by $\varepsilon\mathcal J_m(S_\ast)$.  Since
$\varepsilon$ is arbitrary, for each fixed real $S_\ast$ we obtain
\begin{equation}
 \mathcal I_{f,m}(S_\ast)=\mathcal J_m(S_\ast)(1+o_{f,S_\ast}(1)).
 \label{eq:fixed-weight-IJ-asymptotic}
\end{equation}
In particular, $\mathcal I_{f,m}(S_+)>0$ for all sufficiently large $m$.

\smallskip
\noindent\emph{Comparison of Mellin parameters.}
Choose $R\ge1$ so large that
$y^{S_--S_+}\le\varepsilon$ for $y\ge R$.  The compact part of $\mathcal J_m(S_-)$ on $[1,R]$ is
$O_{S_-,R}((\log R)^m)$, whereas the lower-bound argument above, applied to
$\mathcal J_m(S_+)$ with any fixed $Y_1>R$, gives
$\mathcal J_m(S_+)\gg_{Y_1,S_+}(\log Y_1)^m$.  Hence the compact part is
$o(\mathcal J_m(S_+))$.  On $[R,\infty)$,
\[
 e^{-2\pi y}y^{S_-}(\log y)^m
 \le
 \varepsilon e^{-2\pi y}y^{S_+}(\log y)^m.
\]
It follows that
$\mathcal J_m(S_-)/\mathcal J_m(S_+)\to0$.  Combining this with
\eqref{eq:fixed-weight-IJ-asymptotic} proves
\eqref{eq:moment-ratio-zero}.
\end{proof}

\subsection{The paritywise root-of-unity limits}

Recall the exact split-Mellin identity from
Lemma~\ref{lem:transition-split-mellin}.  We combine it with the dominance
lemma to prove the fixed-weight theorem.

\begin{proof}[Proof of Theorem~\ref{thm:intro-fixed-weight-large-derivative}]
\smallskip
\noindent\emph{Coefficientwise root-of-unity asymptotics.}
Lemma~\ref{lem:log-moment-dominance} and
\eqref{eq:transition-split-mellin} give
\begin{equation}
 \Lambda_f^{(m)}(2)
 =\eta_{k,m}\mathcal I_{f,m}(k-2)(1+o(1)),
 \label{eq:Lambda2-moment-asymptotic}
\end{equation}
so the endpoint value is eventually nonzero.  For every critical integer
$3\le s\le k-3$, both $s$ and $k-s$ are strictly smaller than $k-2$.
Since only finitely many such integers occur,
Lemma~\ref{lem:log-moment-dominance} gives simultaneously
$\mathcal I_{f,m}(s),\mathcal I_{f,m}(k-s)
=o(\mathcal I_{f,m}(k-2))$.  The split-Mellin identity therefore yields
\begin{equation}
 \frac{\Lambda_f^{(m)}(s)}{\Lambda_f^{(m)}(2)}\longrightarrow0
 \qquad(3\le s\le k-3).
 \label{eq:interior-critical-ratio}
\end{equation}

Using
$\Lambda_f^{(m)}(k-2)=(-1)^{k/2+m}\Lambda_f^{(m)}(2)$ and
$i^{4-k}=(-1)^{k/2}$, the two endpoint terms combine as
\[
 (k-2)\Lambda_f^{(m)}(2)\bigl(z^{k-3}+(-1)^m z\bigr).
\]
By \eqref{eq:interior-critical-ratio}, all other normalized coefficients tend
to zero.  Thus, for all sufficiently large $m$, the following normalized
reduced odd polynomial is well defined:
\begin{equation}
 \widehat Q_{f,m}(z):=\frac{Q_{f,m}^{-}(z)}{(k-2)\Lambda_f^{(m)}(2)z}.
 \label{eq:fixed-weight-normalized-reduced}
\end{equation}
It satisfies \eqref{eq:intro-fixed-weight-root-unity-limit}.

\smallskip
\noindent\emph{Localization and unit-circle rigidity.}
Fix one parity class of $m$.  Along this subsequence the target polynomial
$z^{k-4}+(-1)^m$ is fixed and has only simple zeros on $\T$.  Around each
limiting root $\zeta$, choose a small annular sector
\[
 U_\zeta
 =\left\{\zeta e^{u+i\vartheta}:
   u,\vartheta\in\R,\ |u|<\delta,\ |\vartheta|<\delta\right\},
\]
with the sectors pairwise disjoint.  Choose $\delta$ so small that each
$U_\zeta$ is a piecewise smooth Jordan domain, its closure contains no other
limiting root, and its boundary is disjoint from the zeros of
$z^{k-4}+(-1)^m$.

In this fixed degree, coefficientwise convergence is uniform on compact
subsets.  Hence the convergence to $z^{k-4}+(-1)^m$ is uniform on the finite
union $\bigcup_\zeta\partial U_\zeta$, where the limiting polynomial has a
positive minimum.  Rouch\'e's theorem therefore gives exactly one zero,
counted with multiplicity, of $\widehat Q_{f,m}$ in each $U_\zeta$ for all
sufficiently large $m$ of the chosen parity.  Since the reduced polynomial has
degree $k-4$, these sectors account for all of its zeros.

Because $|\zeta|=1$, each sector is invariant under
$z\mapsto1/\overline z$: if $z=\zeta e^{u+i\vartheta}$ lies in $U_\zeta$,
then $1/\overline z=\zeta e^{-u+i\vartheta}$ lies there as well.
Lemma~\ref{lem:Qodd-real-reciprocal} shows that $Q_{f,m}^{-}$ has real
coefficients and satisfies
\[
 Q_{f,m}^{-}(z)=(-1)^m z^{k-2}Q_{f,m}^{-}(1/z).
\]
After division by the nonzero real scalar
$(k-2)\Lambda_f^{(m)}(2)$ and by $z$, the reduced polynomial satisfies
\[
 \widehat Q_{f,m}(z)=(-1)^m z^{k-4}\widehat Q_{f,m}(1/z).
\]
Thus its nonzero root set is invariant under conjugation and inversion, hence
under $z\mapsto1/\overline z$.  The unique zero in each invariant sector must
be fixed by this map; otherwise its image would be a second zero in the same
sector.  Consequently every zero of $\widehat Q_{f,m}$ lies on $\T$.  The
Rouch\'e count also shows that these zeros are simple.  Repeating the argument
for the other parity covers the full sequence of sufficiently large $m$.

Finally, the coefficient of $z$ in $Q_{f,m}^{-}$ is
$(k-2)(-1)^m\Lambda_f^{(m)}(2)\ne0$, so the forced zero at the origin is
simple.  This proves the fixed-$f$ assertion.

\smallskip
\noindent\emph{Uniformity in the eigenform.}
If $S_k(\SLZ)=0$, the uniform assertion is vacuous.  Otherwise, simultaneous
diagonalization of the commuting Hecke operators and multiplicity one imply
that there are only finitely many normalized simultaneous Hecke eigenforms in
$S_k(\SLZ)$.  Taking the maximum of their finitely many thresholds gives an
integer $m_0(k)$ that works uniformly over all of them.
\end{proof}

\subsection{Finite exceptions}

Together, the fixed-weight and large-weight theorems leave only finitely many
weight--derivative pairs to examine.

\begin{proof}[Proof of Corollary~\ref{cor:finite-pair-exception}]
Let $K_0$ be the cutoff in Theorem~\ref{thm:intro-main}, and let
\[
 \mathcal W=\{k<K_0:\ k\text{ even and }S_k(\SLZ)\ne0\}.
\]
For each $k\in\mathcal W$, Theorem~\ref{thm:intro-fixed-weight-large-derivative}
provides an integer $m_0(k)$ that works uniformly for all normalized Hecke
eigenforms in $S_k(\SLZ)$.  Hence the only pairs not covered by the two main
asymptotic theorems belong to
\[
 \mathcal P=\{(k,m): k\in\mathcal W,\ 0\le m<m_0(k)\},
\]
which is finite.  As noted at the end of the proof of
Theorem~\ref{thm:intro-fixed-weight-large-derivative}, each fixed weight in
$\mathcal W$ admits only finitely many normalized simultaneous Hecke
eigenforms.  Thus only finitely many triples $(k,f,m)$ can remain exceptional.
\end{proof}

\appendix
\section{Technical estimates for the upper endpoint}
\label{sec:bell-appendix}

This appendix proves Proposition~\ref{prop:slow-first-order-edge} in four
steps.  Lemma~\ref{lem:bell-coefficient-seminorm} supplies a stable
coefficient seminorm.  Lemma~\ref{lem:normalized-bell-coeff-derivative}
controls the derivative of the resulting normalized coefficient.
Lemma~\ref{lem:slow-bell-expansion} then yields the Gamma-derivative
expansion, and Lemma~\ref{lem:slow-L-factor-negligible} removes the
$L$-factor.  The resulting logarithmic derivative is integrated along the
right edge.

\subsection{Gamma derivatives and the negligible \texorpdfstring{$L$}{L}-factor}

Recall that
\[
 \begin{aligned}
 G(s)&=(2\pi)^{-s}\Gamma(s),
 &\psi(s)&=\frac{\Gamma'(s)}{\Gamma(s)},\\
 \lambda(s)&=\frac{G'(s)}{G(s)}=\psi(s)-\log(2\pi),
 &g_j(s)&=\frac{G^{(j)}(s)}{G(s)}.
 \end{aligned}
\]
For $B>0$, write
\[
 I_{k,B}:=[k-1-B\log k,k-1].
\]
For $\mu\in\Z_{\ge1}$, $\lambda>0$, and a formal power series
$F(t)=\sum_{\ell\ge0}F_\ell t^\ell$, set
\[
 \begin{aligned}
 (\mu)_0&=1,\qquad
 (\mu)_\ell=\mu(\mu-1)\cdots(\mu-\ell+1)
 \quad(1\le \ell\le \mu),\\
 \|F\|_{\mu,\lambda}&:=
 \sum_{0\le \ell\le\mu}|F_\ell|\frac{(\mu)_\ell}{\lambda^\ell}.
 \end{aligned}
\]
When $\lambda=\lambda(s)>0$, also define
\[
 \mathcal C_{\mu,s}(F)
 :=\frac{\mu![t^\mu]e^{\lambda(s)t}F(t)}{\lambda(s)^\mu}
 =\sum_{0\le\ell\le\mu}F_\ell\frac{(\mu)_\ell}{\lambda(s)^\ell}.
\]
Finally, put
\[
 \mathscr U_s(t)=\frac{\psi'(s)}2t^2,
 \qquad
 \mathscr H_s(t)=\sum_{\ell\ge3}\frac{\psi^{(\ell-1)}(s)}{\ell!}t^\ell.
\]
The normalization by $\lambda(s)^\mu$ makes the falling-factorial weights
natural: coefficient extraction from $e^{\lambda(s)t}$ sends $t^\ell$ to
$(\mu)_\ell/\lambda(s)^\ell$.

\begin{lemma}[Coefficient seminorm for Bell remainders]
\label{lem:bell-coefficient-seminorm}
Fix $B,\mathfrak b>0$.  Uniformly for integers $\mu$ with
$1\le\mu\le \mathfrak b\log k$ and $s\in I_{k,B}$, one has
$\lambda(s)>0$ for all sufficiently large $k$.  For every formal power series
$F$,
\begin{equation}
 |\mathcal C_{\mu,s}(F)|\le \|F\|_{\mu,\lambda(s)},
 \label{eq:bell-seminorm-coeff-bound}
\end{equation}
and, for any two formal power series $F_1$ and $F_2$,
\begin{equation}
 \|F_1F_2\|_{\mu,\lambda(s)}
 \le \|F_1\|_{\mu,\lambda(s)}\|F_2\|_{\mu,\lambda(s)}.
 \label{eq:bell-seminorm-submultiplicative}
\end{equation}
Moreover, uniformly in the same range,
\begin{align}
 \|\mathscr H_s\|_{\mu,\lambda(s)}&=O_{B,\mathfrak b}(k^{-2}),
 &\|\partial_s\mathscr H_s\|_{\mu,\lambda(s)}&=O_{B,\mathfrak b}(k^{-3}),
 \label{eq:bell-R-seminorm}\\
 \left\|e^{\mathscr H_s}-1\right\|_{\mu,\lambda(s)}&=O_{B,\mathfrak b}(k^{-2}),
 &\left\|\partial_s(e^{\mathscr H_s})\right\|_{\mu,\lambda(s)}&=O_{B,\mathfrak b}(k^{-3}),
 \label{eq:bell-expR-seminorm}\\
 \|e^{\mathscr U_s}-1-\mathscr U_s\|_{\mu,\lambda(s)}
 &=O_{B,\mathfrak b}(k^{-2}).
 \label{eq:bell-quadratic-tail-seminorm}
\end{align}
\end{lemma}

\begin{proof}
The coefficient bound \eqref{eq:bell-seminorm-coeff-bound} follows from
\[
 \frac{\mu![t^\mu]e^{\lambda t}t^\ell}{\lambda^\mu}
 =\frac{(\mu)_\ell}{\lambda^\ell}.
\]
For submultiplicativity, write
$F_1(t)=\sum_a A_at^a$ and $F_2(t)=\sum_b B_bt^b$.  Since
$(\mu)_{a+b}\le(\mu)_a(\mu)_b$ whenever $a+b\le\mu$, summing the
convolution coefficients of $F_1F_2$ gives
\eqref{eq:bell-seminorm-submultiplicative}.

In the indicated range, $\lambda(s)\asymp\log k$, hence
$\mu/\lambda(s)\le K_{B,\mathfrak b}$ for a constant depending only on
$B$ and $\mathfrak b$.  For $j\ge1$, the polygamma series is
\[
 \psi^{(j)}(s)=(-1)^{j+1}j!\sum_{n\ge0}(s+n)^{-j-1}.
\]
The summand is positive and decreasing, so the integral comparison gives, for
$\ell\ge3$,
\[
\begin{aligned}
 \left|\frac{\psi^{(\ell-1)}(s)}{\ell!}\right|
 &=\frac1\ell\sum_{n\ge0}(s+n)^{-\ell}\\
 &\le \frac1\ell\left(s^{-\ell}
       +\int_0^\infty(s+x)^{-\ell}\,dx\right)\\
 &=\frac{s^{-\ell}}\ell
   +\frac{s^{1-\ell}}{\ell(\ell-1)}
 \le \frac{C}{\ell s^{\ell-1}}.
\end{aligned}
\]
Therefore
\[
 \|\mathscr H_s\|_{\mu,\lambda(s)}
 \le \sum_{\ell\ge3}\frac{C}{\ell s^{\ell-1}}K_{B,\mathfrak b}^\ell
 =O_{B,\mathfrak b}(k^{-2}).
\]
After differentiating a coefficient once,
\[
 \left|\frac{\psi^{(\ell)}(s)}{\ell!}\right|
 =\sum_{n\ge0}(s+n)^{-\ell-1}
 \le s^{-\ell-1}+\frac{s^{-\ell}}\ell.
\]
Since
\[
 \frac{(\mu)_\ell}{\lambda(s)^\ell}
 \le K_{B,\mathfrak b}^\ell,
\]
summation over $\ell\ge3$ gives
$\|\partial_s\mathscr H_s\|_{\mu,\lambda(s)}=O_{B,\mathfrak b}(k^{-3})$.
The estimates for $e^{\mathscr H_s}-1$ and its $s$-derivative follow from
submultiplicativity and the exponential series in the seminorm.  Finally,
\[
 \|\mathscr U_s\|_{\mu,\lambda(s)}
 =\frac{\psi'(s)}2\frac{\mu(\mu-1)}{\lambda(s)^2}
 =O_{B,\mathfrak b}(k^{-1}),
\]
and another application of submultiplicativity and the exponential series
gives \eqref{eq:bell-quadratic-tail-seminorm}.
\end{proof}

\begin{lemma}[Normalized Bell coefficient and its $s$-derivative]
\label{lem:normalized-bell-coeff-derivative}
Fix $B,\mathfrak b>0$.  Uniformly for integers $\mu$ with
$1\le\mu\le \mathfrak b\log k$ and $s\in I_{k,B}$, put
\[
 \mathfrak u_{\mu,s}:=\mathcal C_{\mu,s}(\mathscr U_s)
 =\frac{\psi'(s)\mu(\mu-1)}{2\lambda(s)^2}.
\]
Then
\begin{align}
 \mathcal C_{\mu,s}(e^{\mathscr U_s}e^{\mathscr H_s})
 &=1+\mathfrak u_{\mu,s}+O_{B,\mathfrak b}(k^{-2}),
 \label{eq:normalized-bell-coeff}\\
 \frac{d}{ds}\mathcal C_{\mu,s}(e^{\mathscr U_s}e^{\mathscr H_s})
 &=\mathfrak u'_{\mu,s}+O_{B,\mathfrak b}(k^{-3}).
 \label{eq:normalized-bell-coeff-derivative}
\end{align}
The first estimate also shows that
$\mathcal C_{\mu,s}(e^{\mathscr U_s}e^{\mathscr H_s})>0$ uniformly for all
sufficiently large $k$.  Consequently its real logarithm is well defined and
\[
 \frac{d}{ds}\log \mathcal C_{\mu,s}(e^{\mathscr U_s}e^{\mathscr H_s})
 =O_{B,\mathfrak b}(k^{-2}).
\]
\end{lemma}

\begin{proof}
All unlabelled seminorms below use the weights
$(\mu)_\ell/\lambda(s)^\ell$.  When estimating a coefficientwise derivative
$\partial_sF_s$, we hold these weights fixed.  Their $s$-dependence is restored
separately by \eqref{eq:Cmu-s-derivative}.
Lemma~\ref{lem:bell-coefficient-seminorm} gives
$\|\mathscr H_s\|=O(k^{-2})$ and
$\|\partial_s\mathscr H_s\|=O(k^{-3})$.  Also
\[
 \|\mathscr U_s\|=O(k^{-1}),\qquad
 \|\partial_s\mathscr U_s\|=O(k^{-2}),
\]
because $\psi'(s)=O(k^{-1})$, $\psi''(s)=O(k^{-2})$, and
$\mu/\lambda(s)=O_{B,\mathfrak b}(1)$.  By submultiplicativity,
\[
 e^{\mathscr U_s}e^{\mathscr H_s}=1+\mathscr U_s+\mathscr W_s,
 \qquad
 \|\mathscr W_s\|=O_{B,\mathfrak b}(k^{-2}),
 \qquad
 \|\partial_s\mathscr W_s\|=O_{B,\mathfrak b}(k^{-3}).
\]
Indeed, $e^{\mathscr U_s}=1+\mathscr U_s+O(k^{-2})$ and
$\partial_s(e^{\mathscr U_s}-1-\mathscr U_s)=O(k^{-3})$ in the seminorm,
whereas $e^{\mathscr H_s}=1+O(k^{-2})$ and
$\partial_s e^{\mathscr H_s}=O(k^{-3})$ by
Lemma~\ref{lem:bell-coefficient-seminorm}.  Equation
\eqref{eq:normalized-bell-coeff} now follows from
\eqref{eq:bell-seminorm-coeff-bound}.

For any coefficientwise differentiable family of formal power series $F_s$,
direct differentiation of the coefficient formula gives
\begin{equation}
 \frac{d}{ds}\mathcal C_{\mu,s}(F_s)
 =\mathcal C_{\mu,s}(\partial_sF_s)
 -\frac{\lambda'(s)}{\lambda(s)}\mathcal C_{\mu,s}(\mathcal E F_s),
 \qquad \mathcal E=t\partial_t.
 \label{eq:Cmu-s-derivative}
\end{equation}
For $F_s=\mathscr U_s$, this is exactly $\mathfrak u'_{\mu,s}$, while for
$F_s=1$ both terms vanish.  For $F_s=\mathscr W_s$, the first term is
$O(k^{-3})$.  If
$\mathscr W_s(t)=\sum_\ell \mathfrak w_\ell(s)t^\ell$, then only
coefficients with $\ell\le\mu$ enter $\mathcal C_{\mu,s}$, and hence
\[
 \|\mathcal E\mathscr W_s\|_{\mu,\lambda(s)}
 \le\mu\|\mathscr W_s\|_{\mu,\lambda(s)}.
\]
Thus
\[
 |\mathcal C_{\mu,s}(\mathcal E\mathscr W_s)|
 \le \mu\|\mathscr W_s\|_{\mu,\lambda(s)}
 =O_{B,\mathfrak b}((\log k)k^{-2}),
\]
whereas $\lambda'(s)/\lambda(s)=O((k\log k)^{-1})$.  Applying
\eqref{eq:Cmu-s-derivative} to $1+\mathscr U_s+\mathscr W_s$ proves
\eqref{eq:normalized-bell-coeff-derivative}.  Since
$\mathfrak u_{\mu,s}=O(k^{-1})$ and
$\mathfrak u'_{\mu,s}=O(k^{-2})$, the logarithmic derivative bound follows.
\end{proof}

\begin{lemma}[Uniform Gamma-derivative expansion]
\label{lem:slow-bell-expansion}
Fix $B,\mathfrak b>0$.  Uniformly for integers $\mu$ with
$1\le \mu\le \mathfrak b\log k$ and $s\in I_{k,B}$,
\begin{equation}
 g_\mu(s)=\lambda(s)^\mu\left(1+
 \frac{\psi'(s)\mu(\mu-1)}{2\lambda(s)^2}+O_{B,\mathfrak b}(k^{-2})\right),
 \label{eq:slow-bell-expansion}
\end{equation}
and
\begin{equation}
 \frac{d}{ds}\log g_\mu(s)
 =\frac{\mu\psi'(s)}{\lambda(s)}+O_{B,\mathfrak b}(k^{-2}).
 \label{eq:slow-bell-log-derivative}
\end{equation}
\end{lemma}

\begin{proof}
The Taylor expansion of $\log G$ at $s$ gives
\begin{equation}
 \log\frac{G(s+t)}{G(s)}
 =\lambda(s)t+\frac{\psi'(s)}2t^2+\mathscr H_s(t).
 \label{eq:slow-bell-remainder}
\end{equation}
Since
\[
 \frac{G(s+t)}{G(s)}=\sum_{\nu\ge0}g_\nu(s)\frac{t^\nu}{\nu!},
\]
we have
\begin{equation}
 \frac{g_\mu(s)}{\lambda(s)^\mu}
 =\mathcal C_{\mu,s}(e^{\mathscr U_s}e^{\mathscr H_s}).
 \label{eq:slow-bell-coeff}
\end{equation}
Lemma~\ref{lem:normalized-bell-coeff-derivative} gives
\[
 \frac{g_\mu(s)}{\lambda(s)^\mu}
 =1+\frac{\psi'(s)\mu(\mu-1)}{2\lambda(s)^2}
 +O_{B,\mathfrak b}(k^{-2}),
\]
which proves \eqref{eq:slow-bell-expansion}.  In particular,
$g_\mu(s)>0$ uniformly in this range for all sufficiently large $k$, so its
real logarithm is well defined.  Moreover,
\[
 \frac{d}{ds}\log\frac{g_\mu(s)}{\lambda(s)^\mu}
 =\frac{d}{ds}\log\mathcal C_{\mu,s}(e^{\mathscr U_s}e^{\mathscr H_s})
 =O_{B,\mathfrak b}(k^{-2}).
\]
Since $\lambda'(s)=\psi'(s)$, this gives
\[
 \frac{d}{ds}\log g_\mu(s)
 =\mu\frac{\lambda'(s)}{\lambda(s)}+O_{B,\mathfrak b}(k^{-2})
 =\frac{\mu\psi'(s)}{\lambda(s)}+O_{B,\mathfrak b}(k^{-2}),
\]
which is \eqref{eq:slow-bell-log-derivative}.
\end{proof}

\begin{lemma}[Negligibility of the $L$-factor near the right edge]
\label{lem:slow-L-factor-negligible}
Fix $B,\mathfrak b>0$.  There exists $c_{B,\mathfrak b}>0$ such that,
uniformly for integers $\mu$ with $1\le\mu\le \mathfrak b\log k$,
$s\in I_{k,B}$, and normalized level-one Hecke eigenforms $f$, one has
\begin{equation}
 \mathcal D_{f,\mu}(s)
 =g_\mu(s)\left(1+O_{B,\mathfrak b}(2^{-c_{B,\mathfrak b}k})\right).
 \label{eq:slow-L-factor-negligible}
\end{equation}
\end{lemma}

\begin{proof}
\smallskip
\noindent\emph{Derivatives of the Dirichlet series.}
Recall the normalized Fourier coefficients $\rho_f(n)$ from
Section~\ref{sec:right-edge-estimates}, defined by
$a_f(n)=\rho_f(n)n^{(k-1)/2}$.  Deligne's bound gives
$|\rho_f(n)|\le \tau(n)$, where $\tau$ is the divisor function.  Hence,
uniformly in $f$ and in the indicated range,
\begin{equation}
 L_f(s)-1=O_{B,\mathfrak b}(2^{-c_{B,\mathfrak b}k}),
 \qquad
 L_f^{(j)}(s)=O_{B,\mathfrak b}(2^{-c_{B,\mathfrak b}k})
 \quad(1\le j\le \mathfrak b\log k).
 \label{eq:slow-L-derivative-small}
\end{equation}
For $j=0$, set $T_0(s)=L_f(s)-1$, and for $j\ge1$, set
$T_j(s)=L_f^{(j)}(s)$.  Then
\[
 |T_j(s)|\le
 \sum_{n\ge2}\tau(n)(\log n)^j n^{(k-1)/2-s}.
\]
For $0\le j\le \mathfrak b\log k$ and $x=\log n\ge\log2$, put
\[
 F_j(x)=j\log x-\frac{k}{12}x.
\]
For all sufficiently large $k$,
\[
 F_j'(x)=\frac{j}{x}-\frac{k}{12}
 \le \frac{\mathfrak b\log k}{\log2}-\frac{k}{12}<0.
\]
Moreover, $F_j(\log2)<0$.  Thus $F_j(x)<0$ for all $x\ge\log2$, which gives
\begin{equation}
 (\log n)^j\le n^{k/12}\qquad(n\ge2).
 \label{eq:log-factor-absorbed}
\end{equation}
Since $s\ge k-1-B\log k$, the exponent
$(k-1)/2-s+k/12$ is at most $-k/3$ for large $k$.  Using
$\tau(n)\le n$ and enlarging $k$ once more,
\[
 |T_j(s)|\ll\sum_{n\ge2}n^{-k/4}=O(2^{-c k}),
\]
uniformly in $j$, $s$, and $f$.

\smallskip
\noindent\emph{Transfer to $\mathcal D_{f,\mu}$.}
By definition,
\[
 \mathcal D_{f,\mu}(s)=\frac{(G(s)L_f(s))^{(\mu)}}{G(s)}.
\]
Leibniz' rule therefore gives
\[
 \mathcal D_{f,\mu}(s)
 =\sum_{j=0}^{\mu}\binom{\mu}{j}g_{\mu-j}(s)L_f^{(j)}(s).
\]
The term $j=0$ equals
$g_\mu(s)(1+O_{B,\mathfrak b}(2^{-c_{B,\mathfrak b}k}))$.  For the remaining terms,
Lemma~\ref{lem:slow-bell-expansion}, together with $g_0=1$, gives
\[
 |g_{\mu-j}(s)|\le C_{B,\mathfrak b}\lambda(s)^{\mu-j},
 \qquad
 |g_\mu(s)|\ge C_{B,\mathfrak b}^{-1}\lambda(s)^\mu
\]
for all large $k$, where $\lambda(s)\asymp\log k$.  Hence
\begin{align*}
 \frac1{|g_\mu(s)|}
 \sum_{j=1}^{\mu}\binom{\mu}{j}|g_{\mu-j}(s)|\,|L_f^{(j)}(s)|
 &\ll_{B,\mathfrak b}2^{-c k}
   \sum_{j=1}^{\mu}\binom{\mu}{j}\lambda(s)^{-j} \\
 &\ll_{B,\mathfrak b}2^{-c k}
   \left((1+\lambda(s)^{-1})^\mu-1\right) \\
 &\ll_{B,\mathfrak b}2^{-c' k},
\end{align*}
because $\mu/\lambda(s)=O_{\mathfrak b}(1)$.  After decreasing the positive
constant $c_{B,\mathfrak b}$ in the statement if necessary, this proves
\eqref{eq:slow-L-factor-negligible}.
\end{proof}

\subsection{Proof of the first-order edge expansion}

Recall the displacement
\[
 \mathfrak d_{k,\mu}
 =\frac{\mu\psi'(k-1)}{\psi(k-1)-\log(2\pi)}
\]
from \eqref{eq:slow-edge-displacement}.

\begin{proof}[Proof of Proposition~\ref{prop:slow-first-order-edge}]
Choose $B=B(R)$ large and put $L=\lfloor B\log k\rfloor$.

\smallskip
\noindent\emph{Local right-edge ratios.}
On $I_{k,B}$, Lemma~\ref{lem:slow-L-factor-negligible} gives
\[
 \mathcal D_{f,\mu}(s)
 =g_\mu(s)(1+O_{B,\mathfrak b}(2^{-c_{B,\mathfrak b}k})).
\]
By Lemma~\ref{lem:slow-bell-expansion},
$g_\mu(s)=\lambda(s)^\mu(1+O_{B,\mathfrak b}(k^{-1}))>0$
uniformly for all sufficiently large $k$.  Since the relative error above is
$o(1)$, it follows that
\[
 \mathcal D_{f,\mu}(s)\ne0
 \qquad (s\in I_{k,B}),
\]
uniformly in $\mu$ and $f$.  In particular,
$c_{\mu,0}=\mathcal D_{f,\mu}(k-1)\ne0$.

By Lemma~\ref{lem:slow-bell-expansion}, specifically
\eqref{eq:slow-bell-log-derivative},
\[
 \frac{d}{ds}\log g_\mu(s)
 =\frac{\mu\psi'(s)}{\lambda(s)}+O_{B,\mathfrak b}(k^{-2}).
\]
Set
\[
 h_{k,\mu}(s)=\frac{\mu\psi'(s)}{\lambda(s)}.
\]
The standard bounds $\psi'(s)=O(k^{-1})$, $\psi''(s)=O(k^{-2})$, and
$\lambda(s)\asymp\log k$ give
\[
 h_{k,\mu}'(s)
 =\mu\frac{\psi''(s)\lambda(s)-\psi'(s)^2}{\lambda(s)^2}
 =O_{B,\mathfrak b}(k^{-2})
\]
uniformly for $1\le\mu\le\mathfrak b\log k$.  Since
$h_{k,\mu}(k-1)=\mathfrak d_{k,\mu}$, it follows that, for
$s=k-1-r$ with $0\le r\le L$,
\[
 \frac{\mu\psi'(s)}{\lambda(s)}
 =\mathfrak d_{k,\mu}+O_{B,\mathfrak b}(r/k^2).
\]
Integrating from $k-1$ to $k-1-r$ gives
\[
 \log\frac{g_\mu(k-1-r)}{g_\mu(k-1)}
 =-r\mathfrak d_{k,\mu}
 +O_{B,\mathfrak b}\left(\frac{r+r^2}{k^2}\right).
\]
Since
\[
 r\mathfrak d_{k,\mu}
 =O_{B,\mathfrak b}\left(\frac{\log k}{k}\right)=o(1),
\]
exponentiation and Lemma~\ref{lem:slow-L-factor-negligible} yield
\begin{equation}
 \frac{\mathcal D_{f,\mu}(k-1-r)}{\mathcal D_{f,\mu}(k-1)}
 =1-r\mathfrak d_{k,\mu}
 +O_{B,\mathfrak b}\left(
 r^2\mathfrak d_{k,\mu}^2+\frac{(1+r)^2}{k^2}+2^{-ck}
 \right)
 \label{eq:slow-D-ratio}
\end{equation}
for $0\le r\le L$.

\smallskip
\noindent\emph{Summation and factorial tails.}
For all sufficiently large $k$, one has $L<M$.  Hence the coefficient of
$w^r$ in $q_{f,\mu,k}$ is $c_{\mu,r}$ throughout the local range
$0\le r\le L$; the halved central coefficient does not occur there.  After
multiplication of \eqref{eq:slow-D-ratio} by $(2\pi w)^r/r!$ and summation over
$0\le r\le L$, the total contribution of the error term is
$o_{\mathfrak b,R}(\mathfrak d_{k,\mu})$.  Indeed, for $|w|\le R$, each
factorial moment
\[
 \sum_{r\ge0} r^j\frac{(2\pi R)^r}{r!}
\]
is bounded by a constant depending only on $j$ and $R$.  Hence the summed
error is
\[
 O_{\mathfrak b,R}(\mathfrak d_{k,\mu}^2+k^{-2}+e^{-c k})
 =o_{\mathfrak b,R}(\mathfrak d_{k,\mu}),
\]
uniformly for $1\le\mu\le\mathfrak b\log k$, because
$\mathfrak d_{k,\mu}\gg(k\log k)^{-1}$.

The same choice of $B$ makes the model factorial tails of $e^{2\pi w}$ and
$2\pi we^{2\pi w}$ over $r\ge L+1$ equal to
$o_{\mathfrak b,R}(\mathfrak d_{k,\mu})$.  Proposition~\ref{prop:transition-global-factorial-tail},
applied with $A=2$ and lower cutoff $L+1$, controls the corresponding tail of
the actual edge polynomial.  Summing the main terms gives
\[
 \sum_{r\ge0}\frac{(2\pi w)^r}{r!}(1-r\mathfrak d_{k,\mu})
 =e^{2\pi w}(1-2\pi\mathfrak d_{k,\mu}w),
\]
which proves \eqref{eq:slow-first-order-q}.  Substituting $w=iz$ and taking
the odd part gives \eqref{eq:slow-first-order-p}.
\end{proof}

\section{Effectivity and the remaining finite region}
\label{sec:effectivity}

In this paper, ``effective'' is used only in the following in-principle sense:
once the auxiliary compact sets, neighborhoods, and Rouch\'e margins have been
fixed and all constants in the analytic estimates have been tracked, the
thresholds in the asymptotic statements can be extracted.  We do not claim
that the present proof supplies practical numerical constants or a useful
numerical value of $K_0$.  Obtaining such a value would require tracking four
groups of constants:
\begin{enumerate}
\item the uniform Dirichlet and Mellin remainders, the Bell--Gamma
estimates, and the factorial tail;
\item the curvature thresholds in the moving-saddle approximation;
\item the Rouch\'e, annular-exclusion, and winding-count margins on the
separated compact sets and in the critical and endpoint neighborhoods;
\item the fixed-weight thresholds $m_0(k)$, including explicit data for the
finitely many normalized eigenforms in the remaining weights.
\end{enumerate}
We do not optimize these constants.  Accordingly,
Corollary~\ref{cor:finite-pair-exception} is qualitative: this paper does not
provide explicit numerical bounds for the finite region $\mathcal P$, optimized
or otherwise, and does not prove that $\mathcal P$ is empty.  No numerical
low-weight verification is used in any theorem.  A rigorous elimination of
the remaining finite region would require explicit constant extraction
followed, if necessary, by certified computation rather than floating-point
evidence alone.

\medskip
\begingroup
\footnotesize
\noindent\textbf{Use of AI-assisted tools.}
For transparency, during the preparation of this manuscript the author used
OpenAI ChatGPT to discuss possible gaps in arguments, review notation and
cross-references, improve the organization and exposition, and edit \LaTeX\
and English prose.  The tool was not used as a source of mathematical results
or as a formal proof verifier.  All proofs, computations, statements, and
references were independently checked by the author, who takes full
responsibility for the manuscript.
\par
\endgroup

\end{document}